\documentclass[10pt,reqno]{amsart}
\usepackage[hypertex]{hyperref}

\usepackage{amsmath,amsthm}
\usepackage{amssymb}
\usepackage{euscript}
\usepackage[all,tips]{xy}

\numberwithin{equation}{subsection}

\allowdisplaybreaks[4]

\newcommand{\sqsp}{\renewcommand{\baselinestretch}{1.1}\tiny\normalsize}

\newcommand{\relphantom[1]}{\mathrel{\phantom{#1}}}

\newtheorem{theorem}[subsection]{Theorem}
\newtheorem{lemma}[subsection]{Lemma}

\newtheorem{corollary}[subsection]{Corollary}

\theoremstyle{definition}
\newtheorem{definition}[subsection]{Definition}
\newtheorem{example}[subsection]{Example}
\newtheorem{remark}[subsection]{Remark}

\newcommand{\bC}{\mathbf{C}}
\newcommand{\bk}{\mathbf{k}}

\newcommand{\bZ}{\mathbf{Z}}

\def\mualphaA{\mu_\alpha^A}
\def\mualphaH{\mu_\alpha^H}

\def\qp{\mathbb{A}^{2|0}_q}
\def\qpalpha{\mathbb{A}^{2|0}_{q,\alpha}}
\def\qf{\mathbb{A}^{0|2}_q}
\def\qfalpha{\mathbb{A}^{0|2}_{q,\alpha}}
\def\qm{\mathbb{A}^{1|1}_q}
\def\qmalpha{\mathbb{A}^{1|1}_{q,\alpha}}
\def\qmtwisted{\mathbb{A}^{1|1}_{-q^{-1}}}

\def\kg{\bk G}
\def\mq{M_q(2)}
\def\mqalpha{\mq_\alpha}
\def\mpq{M_{p,q}(2)}
\def\mqns{M_q(1|1)}
\def\mqnsalpha{\mqns_\alpha}
\def\gl{GL_q(2)}
\def\glalpha{\gl_\alpha}
\def\slq{SL_q(2)}
\def\slqalpha{\slq_\alpha}
\def\uq{U_q(sl_2)}
\def\smalluq{u_q(sl_2)}

\def\uqr{U_q^{(r)}(sl_2)}
\def\uqralpha{\uqr_\alpha}

\def\an{A^{(n)}}
\def\aone{A^{(1)}}
\def\agamma{A(\gamma)}
\def\agammaalpha{\agamma_\alpha}
\def\detq{{\det_q\nolimits}}
\def\rn{R^{\alpha^n}}
\def\rt{R^{\alpha^t}}
\def\ralpha{R^{\alpha}}

\def\abcd{{\begin{pmatrix}a & b\\c & d\end{pmatrix}}}

\def\abcdlambda{{\begin{pmatrix}a & \lambda b\\ \lambda^{-1}c & d\end{pmatrix}}}

\def\Tmatrix{{\begin{pmatrix}T_1^1 & T_1^2\\T_2^1 & T_2^2\end{pmatrix}}}

\def\Tmatrixlambda{{\begin{pmatrix}T_1^1 & \lambda T_1^2\\ \lambda^{-1}T_2^1 & T_2^2\end{pmatrix}}}

\def\mqr{{\begin{pmatrix}
R(a \otimes a) & R(a \otimes b) & R(b \otimes a) & R(b \otimes b)\\
R(a \otimes c) & R(c \otimes b) & R(d \otimes a) & R(b \otimes d)\\
R(c \otimes a) & R(a \otimes d) & R(b \otimes c) & R(d \otimes b)\\
R(c \otimes c) & R(c \otimes d) & R(d \otimes c) & R(d \otimes d)
\end{pmatrix}}}

\def\mr{{\begin{pmatrix}q & 0 & 0 & 0\\
0 & 0 & 1 & 0\\
0 &  1 & q - q^{-1} & 0\\
0 & 0 & 0 & q\\
\end{pmatrix}}}

\def\mpqr{{\begin{pmatrix}q & 0 & 0 & 0\\
0 & 0 & 1 & 0\\
0 & qp^{-1} & q - p^{-1} & 0\\
0 & 0 & 0 & q\\
\end{pmatrix}}}

\def\mqnsmatrix{{\begin{pmatrix}q & 0 & 0 & 0\\
0 & 0 & 1 & 0\\
0 & 1 & q - q^{-1} & 0\\
0 & 0 & 0 & -q^{-1}\\
\end{pmatrix}}}

\def\xymat{{\begin{pmatrix}x\\y\end{pmatrix}}}

\def\alphaxy{{\begin{pmatrix}\alpha(x)\\ \alpha(y)\end{pmatrix}}}

\def\twistedxy{{\begin{pmatrix}\xi x\\ \lambda^{-1}\xi y\end{pmatrix}}}

\DeclareMathOperator{\Hom}{Hom}

\begin{document}

\title[Cobraided Hom-bialgebras and Hom-quantum geometry]{Hom-quantum groups II: cobraided Hom-bialgebras and Hom-quantum geometry}
\author{Donald Yau}

\begin{abstract}
A class of non-associative and non-coassociative generalizations of cobraided bialgebras, called cobraided Hom-bialgebras, is introduced.  The non-(co)associativity in a cobraided Hom-bialgebra is controlled by a twisting map.  Several methods for constructing cobraided Hom-bialgebras are given.  In particular, Hom-type generalizations of FRT quantum groups, including quantum matrices and related quantum groups, are obtained.  Each cobraided Hom-bialgebra comes with solutions of the operator quantum Hom-Yang-Baxter equations, which are twisted analogues of the operator form of the quantum Yang-Baxter equation.  Solutions of the Hom-Yang-Baxter equation can be obtained from comodules of suitable cobraided Hom-bialgebras.  Hom-type generalizations of the usual quantum matrices coactions on the quantum planes give rise to non-associative and non-coassociative analogues of quantum geometry.
\end{abstract}



\address{Department of Mathematics\\
    The Ohio State University at Newark\\
    1179 University Drive\\
    Newark, OH 43055, USA}
\email{dyau@math.ohio-state.edu}

\date{\today}
\maketitle

\sqsp

\section{Introduction}

This paper is part of an on-going effort \cite{yau5,yau6,yau7,yau8} to study twisted, Hom-type generalizations of the various Yang-Baxter equations and related algebraic structures, including Lie bialgebras and quantum groups.  A Hom-type generalization of the Yang-Baxter equation (YBE)  \cite{baxter,baxter2,yang}, called the Hom-Yang-Baxter equation (HYBE), and its relationships to the braid relations and braid group representations \cite{artin2,artin} were studied in \cite{yau5,yau6}.  Hom versions of the classical Yang-Baxter equation \cite{skl1,skl2} and of Drinfel'd's Lie bialgebras \cite{dri83,dri87} were studied in \cite{yau7}.

Hom-type generalizations of Drinfel'd's quasi-triangular (a.k.a.\ braided) bialgebras \cite{dri87}, called quasi-triangular Hom-bialgebras, and of the quantum Yang-Baxter equation (QYBE) \cite{dri85}, called the quantum Hom-Yang-Baxter equations (QHYBEs), were studied in \cite{yau8}.  To simplify the terminology, in this paper we refer to a quasi-triangular Hom-bialgebra as a \emph{braided Hom-bialgebra}.  We use the name \emph{Hom-quantum groups} colloquially to refer to Hom-type generalizations of quantum groups.  Since braided bialgebras are examples of quantum groups, braided Hom-bialgebras form a class of Hom-quantum groups.  As discussed in \cite{yau8}, examples of braided Hom-bialgebras include Hom versions of Drinfel'd's quantum enveloping algebras \cite{dri87}.  These Hom-quantum enveloping algebras are non-associative, non-coassociative, non-commutative, and non-cocommutative.


The main purpose of this paper is to study another class of Hom-quantum groups, called cobraided Hom-bialgebras, generalizing cobraided bialgebras \cite{hay,lt,majid91,sch}.  In the literature, a cobraided bialgebra is also called a dual quasi-triangular bialgebra and a co-quasi-triangular bialgebra.  Cobraided bialgebras are important because, among other properties, they generate lots of solutions of the YBE via their comodules.  Moreover, some cobraided bialgebras have interesting coactions on quantum spaces.  Both of these properties of cobraided bialgebras are generalized to the Hom setting in this paper.  As we discuss below, suitable cobraided Hom-bialgebras generate lots of solutions of the HYBE \cite{yau5,yau6} via their comodules.  Moreover, Hom versions of certain quantum group coactions on the quantum planes lead to \emph{Hom-quantum geometry}, a non-associative and non-coassociative generalization of quantum geometry.

A distinct feature of (co)braided Hom-bialgebras is that they are, in general, non-associative, non-coassociative, non-commutative, and non-cocommutative.  The non-(co)associativity is controlled by a twisting map $\alpha$.  A cobraided bialgebra is an example of a cobraided Hom-bialgebra with $\alpha$ being the identity map.  We will discuss several general methods for constructing classes of cobraided Hom-bialgebras.

Before we describe the results in this paper, let us first recall some basic ideas about Hom-type objects.  Roughly speaking, a Hom-type structure arises when one strategically replaces the identity map in the defining axioms of a classical structure by a general twisting map $\alpha$.  For example, in a Hom-associative algebra \cite{ms}, the multiplication satisfies the Hom-associativity axiom: $\alpha(x)(yz) = (xy)\alpha(z)$.  Similarly, in a Hom-Lie algebra, the bracket satisfies the Hom-Jacobi identity: $[[x,y],\alpha(z)] + [[z,x],\alpha(y)] + [[y,z],\alpha(x)] = 0$.  Hom-Lie algebras were introduced in \cite{hls} to describe the structures on some $q$-deformations of the Witt and the Virasoro algebras.  Earlier precursors of Hom-Lie algebras can be found in \cite{hu,liu}.  Hom-Lie algebras are closely related to deformed vector fields \cite{ama,hls,ls,ls2,ls3,rs,ss} and number theory \cite{larsson}.  Other papers concerning Hom-associative algebras, Hom-Lie algebras, and related Hom-type structures are \cite{ams,cg,fg,fg2,gohr,jl,ms2,ms3,ms4,yau,yau2,yau3,yau4,yau5,yau6,yau7,yau8}.


We now describe the main results of this paper.  Precise definitions, statements of results, and proofs are given in later sections.

A cobraided bialgebra $(A,R)$ \cite{hay,lt,majid91,sch} consists of a bialgebra $A$ and a bilinear form $R \in \Hom(A^{\otimes 2},\bk)$, called the universal $R$-form or the cobraiding form, satisfying three axioms.  In section ~\ref{sec:dqt} we generalize cobraided bialgebras to cobraided Hom-bialgebras (Definition ~\ref{def:dqt}).  The cobraiding form $R$ in a cobraided bialgebra satisfies the operator quantum Yang-Baxter equation (OQYBE) \eqref{oqybe}.  Likewise, we show that in a cobraided Hom-bialgebra, the bilinear form $R$ satisfies \emph{two} Hom-type generalizations of the OQYBE, called the operator quantum Hom-Yang-Baxter equations (OQHYBEs) (Theorem ~\ref{thm:oqhybe}).  The OQHYBEs are the operator forms of the QHYBEs studied in \cite{yau8}.  Some alternative characterizations of the axioms of a cobraided Hom-bialgebra are given at the end of section ~\ref{sec:dqt} (Theorems ~\ref{thm:axiom1} and \ref{thm:axiom2}).

In sections ~\ref{sec:twist} - \ref{sec:duality} we give some general procedures by which cobraided Hom-bialgebras can be constructed and provide specific examples of cobraided Hom-bialgebras.  The first such procedure yields a family (usually infinite) of cobraided Hom-bialgebras $A_\alpha$ from each cobraided bialgebra $A$, where $\alpha$ runs through the bialgebra endomorphisms on $A$ (Theorem ~\ref{thm:twist}).  The cobraided Hom-bialgebra $A_\alpha$ is obtained from $A$ by twisting its (co)multiplication along the map $\alpha$, keeping the bilinear form $R$ unchanged.  A twisting procedure along these lines was first used in \cite{yau2} to construct examples of Hom-associative and Hom-Lie algebras (Remark ~\ref{rk2:twist}).  As examples, we apply this twisting procedure to the quantum group $\mq$ of quantum matrices (Example ~\ref{ex:mq}), the quantum general linear group $\gl$ (Example ~\ref{ex:gl}), the quantum special linear group $\slq$ (Example ~\ref{ex:sl}), the $2$-parameter quantum group $\mpq$ (Example ~\ref{ex:mpq}), and the non-standard quantum group $\mqns$ (Example ~\ref{ex:mq'}).  Since these quantum groups are non-commutative and non-cocommutative, the cobraided Hom-bialgebras obtained by twisting their (co)multiplications are simultaneously non-associative, non-coassociative, non-commutative, and non-cocommutative.

One widely used method for constructing quantum groups is the so-called FRT construction \cite{rft}.  Given an $R$-matrix (i.e., a solution of the YBE \eqref{ybe}) $\gamma$ on a finite dimensional vector space $V$, the FRT construction yields a cobraided bialgebra $A(\gamma)$, called an FRT quantum group (Definition ~\ref{def:frt}), together with a non-trivial $A(\gamma)$-comodule structure on $V$ \eqref{frtcomodule}.  We show in section ~\ref{sec:frt} that the first twisting procedure (Theorem ~\ref{thm:twist}) can be applied naturally to FRT quantum groups.  In particular, we construct explicit bialgebra endomorphisms  $\alpha$ on an arbitrary FRT quantum group $A(\gamma)$ (Theorem ~\ref{thm:frt}).  Together with Theorem ~\ref{thm:twist}, this result yields cobraided Hom-bialgebras $A(\gamma)_\alpha$, generalizing the FRT quantum group $A(\gamma)$ (Corollary ~\ref{cor:frt}).  As an example, we observe that the cobraided Hom-bialgebras $\mqalpha$ \eqref{mqalpha}, $\mpq_\alpha$ \eqref{mpqalpha}, and $\mqnsalpha$ \eqref{mq'alpha} can all be obtained this way (Example ~\ref{ex2:mq}).

A second twisting procedure (Theorem ~\ref{thm:twistinj}) applies to a cobraided Hom-bialgebra $A$ in which the twisting map $\alpha$ is injective.  In this case, there is a sequence $\{\an\}_{n \geq 1}$ of cobraided Hom-bialgebras with the same (co)multiplication and twisting map $\alpha$ as $A$.  The cobraided Hom-bialgebra $\an$ is obtained from $A$ by twisting the bilinear form $R$ by $\alpha^n$, i.e., by replacing $R$ by $\rn = R \circ (\alpha^n \otimes \alpha^n)$.  As examples, we apply this twisting procedure (and Theorem ~\ref{thm:twist}) to quantum group bialgebras (Example ~\ref{ex:groupbi}), the anyon-generating quantum groups  \cite{lm,majid92} (Example ~\ref{ex:anyon}), and an integral version of the anyon-generating quantum groups (Example ~\ref{ex:Z}).

In section ~\ref{sec:duality} we study the duality between braided and cobraided Hom-bialgebras.  In particular, we show that the dual of a finite dimensional braided Hom-bialgebra \cite{yau8} (Definition ~\ref{def:braided}) is a cobraided Hom-bialgebra, and vice versa (Theorem ~\ref{thm:dual}).  As examples, we consider two finite dimensional versions of the quantum enveloping algebra $\uq$ \cite{dri87,kr,skl3}, denoted by $\uqr$ (Example ~\ref{ex:uqr}) and $\smalluq$ (Example ~\ref{ex:smalluq}) \cite{lm,majid92}.  Using a twisting procedure in \cite{yau8} (dual to Theorem ~\ref{thm:twist}), we obtain families of finite dimensional braided Hom-bialgebras $\uqr_\alpha$ \eqref{uqralpha} and $\smalluq_\alpha$ \eqref{smalluqalpha}.  Since the finite dimensional braided Hom-bialgebras $\uqr_\alpha$ and $\smalluq_\alpha$ are simultaneously non-associative, non-coassociative, non-commutative, and non-cocommutative, so are their dual cobraided Hom-bialgebras.

Every comodule over a cobraided bialgebra has a canonical solution \eqref{BVV} of the Yang-Baxter equation, induced by the comodule structure and the cobraiding form.  In section ~\ref{sec:hybe} we show that every comodule (Definition ~\ref{def:comodule}) over a cobraided Hom-bialgebra with $\alpha$-invariant $R$ has a canonical solution of the HYBE (Corollary ~\ref{cor:comod}).  Similar to the un-twisted case, a crucial ingredient in establishing this solution of the HYBE (see the proof of Lemma ~\ref{lem1:comod}) is the fact that the bilinear form $R$ in a cobraided Hom-bialgebra satisfies the OQHYBE \eqref{oqhybe}.

In section ~\ref{sec:coaction} we study Hom-type, non-associative and non-coassociative analogues of quantum geometry.  One aspect of quantum geometry consists of coactions of suitable quantum groups on the standard, fermionic, and mixed quantum planes (Definition ~\ref{def:qplanes}).  We consider Hom versions of the usual quantum matrices comodule algebra structures on these quantum planes (Example ~\ref{ex:qsym}).  We first twist each of these quantum planes into an infinite family of Hom-associative algebras, called the (standard, fermionic, or mixed) Hom-quantum planes (Example ~\ref{ex:hqp}).  Then we show that each such Hom-quantum plane is non-trivially a comodule Hom-algebra (Definition ~\ref{def:cha}) over the cobraided Hom-bialgebras $\mqalpha$ or $\mqnsalpha$ (Examples ~\ref{ex:shqsym} - ~\ref{ex:mhqsym}) constructed in section ~\ref{sec:twist}.  Thus, we have multi-parameter families of Hom-type, non-(co)associative twistings of the usual quantum matrices coactions on the various quantum planes.  The non-associativity refers to that of the multiplications in the cobraided Hom-bialgebras $\mqalpha$ and $\mqnsalpha$ and in the Hom-quantum planes.  The non-coassociativity refers to that of the comultiplications in $\mqalpha$ and $\mqnsalpha$ and in the comodule structure map \eqref{homcoass}.

\section{Cobraided Hom-bialgebras and operator QHYBEs}
\label{sec:dqt}

In this section, we define cobraided Hom-bialgebra (Definition ~\ref{def:dqt}), the main object of study in this paper.  We show that each cobraided Hom-bialgebra has solutions of the Hom versions of the operator quantum Yang-Baxter equation (Theorem ~\ref{thm:oqhybe}).  At the end of this section, we give alternative characterizations of some of the defining axioms of a cobraided Hom-bialgebra (Theorems ~\ref{thm:axiom1} and ~\ref{thm:axiom2}).

\subsection{Conventions and notations}

The conventions are the same as in \cite{yau8}.  We work over a fixed associative and commutative ring $\bk$ of characteristic $0$.  Modules, tensor products, and linear maps are all taken over $\bk$.  If $V$ and $W$ are $\bk$-modules, then $\tau \colon V \otimes W \to W \otimes V$ denotes the twist isomorphism, $\tau(v \otimes w) = w \otimes v$.  For a map $\phi \colon V \to W$ and $v \in V$, we sometimes write $\phi(v)$ as $\langle \phi,v\rangle$.  If $\bk$ is a field and $V$ is a $\bk$-vector space, then the linear dual of $V$ is $V^* = \Hom(V,\bk)$.  From now on, whenever the linear dual $V^*$ is in sight, we tacitly assume that $\bk$ is a characteristic $0$ field.

Given a bilinear map $\mu \colon V^{\otimes 2} \to V$ and elements $x,y \in V$, we often write $\mu(x,y)$ as $xy$ and put in parentheses for longer products.  For a map $\Delta \colon V \to V^{\otimes 2}$, we use Sweedler's notation \cite{sweedler} for comultiplication: $\Delta(x) = \sum_{(x)} x_1 \otimes x_2$.  To simplify the typography in computations, we often omit the subscript in $\sum_{(x)}$ and even the summation sign itself.


\begin{definition}
\label{def:homas}
\begin{enumerate}
\item
A \textbf{Hom-associative algebra} \cite{ms} $(A,\mu,\alpha)$ consists of a  $\bk$-module $A$, a bilinear map $\mu \colon A^{\otimes 2} \to A$ (the multiplication), and a linear self-map $\alpha \colon A \to A$ (the twisting map) such that:
\begin{equation}
\label{homassaxioms}
\begin{split}
\alpha \circ \mu &= \mu \circ \alpha^{\otimes 2} \quad \text{(multiplicativity)},\\
\mu \circ (\alpha \otimes \mu) &= \mu \circ (\mu \otimes \alpha) \quad \text{(Hom-associativity)}.
\end{split}
\end{equation}
\item
A \textbf{Hom-coassociative coalgebra} \cite{ms2,ms4} $(C,\Delta,\alpha)$ consists of a $\bk$-module $C$, a linear map $\Delta \colon C \to C^{\otimes 2}$ (the comultiplication), and a linear self-map $\alpha \colon C \to C$ (the twisting map) such that:
\begin{equation}
\label{homcoassaxioms}
\begin{split}
\alpha^{\otimes 2} \circ \Delta &= \Delta \circ \alpha \quad \text{(comultiplicativity)},\\
(\alpha \otimes \Delta) \circ \Delta &= (\Delta \otimes \alpha) \circ \Delta \quad \text{(Hom-coassociativity)}.
\end{split}
\end{equation}
\item
A \textbf{Hom-bialgebra} \cite{ms2,yau3} is a quadruple $(A,\mu,\Delta,\alpha)$ in which $(A,\mu,\alpha)$ is a Hom-associative algebra, $(A,\Delta,\alpha)$ is a Hom-coassociative coalgebra, and the following condition holds:
\begin{equation}
\label{def:hombi}
\Delta \circ \mu = \mu^{\otimes 2} \circ (Id \otimes \tau \otimes Id) \circ \Delta^{\otimes 2}.
\end{equation}
In terms of elements, \eqref{def:hombi} means that
\[
\Delta(ab) = \sum a_1b_1 \otimes a_2b_2
\]
for all $a, b \in A$.
\end{enumerate}
\end{definition}

Observe that a Hom-bialgebra is neither associative nor coassociative, unless of course $\alpha = Id$, in which case we have a bialgebra.  Instead of (co)associativity, in a Hom-bialgebra we have Hom-(co)associativity, $\mu \circ (\alpha \otimes \mu) = \mu \circ (\mu \otimes \alpha)$ and $(\alpha \otimes \Delta) \circ \Delta = (\Delta \otimes \alpha) \circ \Delta$.  So, roughly speaking, the degree of non-(co)associativity in a Hom-bialgebra is measured by how far the twisting map $\alpha$ deviates from the identity map.

\begin{example}[\textbf{Twisting classical structures}]
\label{ex:homas}
\begin{enumerate}
\item
If $(A,\mu)$ is an associative algebra and $\alpha \colon A \to A$ is an algebra morphism, then $A_\alpha = (A,\mu_\alpha,\alpha)$ is a Hom-associative algebra with the twisted multiplication $\mu_\alpha = \alpha \circ \mu$ \cite{yau2}.  Indeed, the Hom-associativity axiom $\mu_\alpha \circ (\alpha \otimes \mu_\alpha) = \mu_\alpha \circ (\mu_\alpha \otimes \alpha)$ is equal to $\alpha^2$ applied to the associativity axiom of $\mu$.  Likewise, both sides of the multiplicativity axiom $\alpha \circ \mu_\alpha = \mu_\alpha \circ \alpha^{\otimes 2}$ are equal to $\alpha^2 \circ \mu$.
\item
Dually, if $(C,\Delta)$ is a coassociative coalgebra and $\alpha \colon C \to C$ is a coalgebra morphism, then $C_\alpha = (C,\Delta_\alpha,\alpha)$ is a Hom-coassociative coalgebra with the twisted comultiplication $\Delta_\alpha = \Delta \circ \alpha$.
\item
A bialgebra is exactly a Hom-bialgebra with $\alpha = Id$.  More generally, combining the previous two cases, if $(A,\mu,\Delta)$ is a bialgebra and $\alpha \colon A \to A$ is a bialgebra morphism, then $A_\alpha = (A,\mu_\alpha,\Delta_\alpha,\alpha)$ is a Hom-bialgebra.  The compatibility condition \eqref{def:hombi} for $\Delta_\alpha = \Delta \circ \alpha$ and $\mu_\alpha = \alpha \circ \mu$ is straightforward to check.\qed
\end{enumerate}
\end{example}

\begin{example}[\textbf{Duality}]
\label{ex:duality}
\begin{enumerate}
\item
Let $(C,\Delta,\alpha)$ be a Hom-coassociative coalgebra and $C^*$ be the linear dual of $C$.  Then we have a Hom-associative algebra $(C^*,\Delta^*,\alpha^*)$, where
\begin{equation}
\label{deltadual}
\langle \Delta^*(\phi,\psi),x\rangle = \langle \phi \otimes \psi, \Delta(x)\rangle \quad \text{and} \quad
\alpha^*(\phi) = \phi \circ \alpha
\end{equation}
for all $\phi, \psi \in C^*$ and $x \in C$.  This is checked in exactly the same way as for (co)associative algebras \cite[2.1]{abe}, as was done in \cite[Corollary 4.12]{ms4}.
\item
Likewise, suppose that $(A,\mu,\alpha)$ is a finite dimensional Hom-associative algebra. Then $(A^*,\mu^*,\alpha^*)$ is a Hom-coassociative coalgebra, where
\begin{equation}
\label{mudual}
\langle \mu^*(\phi), x \otimes y\rangle = \langle \phi, \mu(x,y)\rangle  \quad \text{and} \quad
\alpha^*(\phi) = \phi \circ \alpha
\end{equation}
for all $\phi \in A^*$ and $x,y \in A$ \cite[Corollary 4.12]{ms4}.  In what follows, whenever $\mu^* \colon A^* \to A^* \otimes A^*$ is in sight, we tacitly assume that $A$ is finite dimensional.
\item
Combining the previous two examples, suppose that $(A,\mu,\Delta,\alpha)$ is a finite dimensional Hom-bialgebra.  Then so is $(A^*,\Delta^*,\mu^*,\alpha^*)$, where $\Delta^*$, $\mu^*$, and $\alpha^*$ are defined as in \eqref{deltadual} and \eqref{mudual}.\qed
\end{enumerate}
\end{example}

We are now ready to give the main definition of this paper.


\begin{definition}
\label{def:dqt}
A \textbf{cobraided Hom-bialgebra} is a quintuple $(A,\mu,\Delta,\alpha,R)$ in which $(A,\mu,\Delta,\alpha)$ is a Hom-bialgebra and $R$ is a bilinear form on $A$ (i.e., $R \in \Hom(A^{\otimes 2},\bk)$), satisfying the following three axioms for $x,y,z \in A$:
\begin{subequations}
\label{dqtaxioms}
\begin{align}
R(xy \otimes \alpha(z)) &= \sum_{(z)} R(\alpha(x) \otimes z_1)R(\alpha(y) \otimes z_2),\label{axiom1}\\
R(\alpha(x) \otimes yz) &= \sum_{(x)} R(x_1 \otimes \alpha(z))R(x_2 \otimes \alpha(y)),\label{axiom2}\\
\sum_{(x)(y)} y_1x_1R(x_2 \otimes y_2) &= \sum_{(x)(y)} R(x_1 \otimes y_1)x_2y_2.\label{axiom3}
\end{align}
\end{subequations}
Here $\Delta(x) = \sum_{(x)} x_1 \otimes x_2$.  We call $R$ the \textbf{Hom-cobraiding form}.  We say that $R$ is \textbf{$\alpha$-invariant} if $R = R \circ \alpha^{\otimes 2}$.
\end{definition}

The notion of $\alpha$-invariance will play a major role in section ~\ref{sec:hybe}.

As the terminology suggests, cobraided Hom-bialgebras are dual to braided Hom-bialgebras \cite{yau8}, at least in the finite dimensional case.  We will prove this in Theorem ~\ref{thm:dual}.  The axioms \eqref{axiom1} and \eqref{axiom2} can be interpreted as saying that $R$ is a Hom version of a bialgebra bicharacter \cite[p.51]{majid}. Axiom \eqref{axiom3} says that the multiplication $\mu$ is ``almost commutative."

\begin{example}
\label{ex:dqtbialgebra}
A \textbf{cobraided bialgebra} (\cite{hay,lt,majid91,sch}, \cite[VIII.5]{kassel}, \cite[2.2]{majid}) $(A,R)$ consists of a bialgebra $A$ and a bilinear form $R$ on $A$, called the \textbf{universal $R$-form} or the \textbf{cobraiding form}, such that the three axioms \eqref{dqtaxioms} hold for $\alpha = Id$.  In the literature, a cobraided bialgebra is usually assumed to have a unit and a counit, and the cobraiding form $R$ is usually required to be invertible under the convolution product in $\Hom(A^{\otimes 2},\bk)$.  In this paper, we never have to use the invertibility of the cobraiding form in a cobraided bialgebra.  We will provide more examples of cobraided Hom-bialgebras in the next few sections.
\qed
\end{example}

We now study the connection between cobraided Hom-bialgebras and Hom versions of the operator quantum Yang-Baxter equation.  Recall that the \textbf{quantum Yang-Baxter equation} (QYBE) \cite{dri85} states
\begin{equation}
\label{qybe}
R_{12}R_{13}R_{23} = R_{23}R_{13}R_{12}.
\end{equation}
If $(H,R)$ is a braided bialgebra \cite{dri87}, then $R \in H^{\otimes 2}$ satisfies the QYBE, where $R_{12} = R \otimes 1$, $R_{23} = 1 \otimes R$, and $R_{13} = (\tau \otimes Id)(R_{23})$.  See, e.g., \cite[Lemma 2.1.4]{majid}. Dually, in a cobraided bialgebra $(A,R)$ (Example ~\ref{ex:dqtbialgebra}), the cobraiding form satisfies the \textbf{operator quantum Yang-Baxter equation} (OQYBE)
\begin{equation}
\label{oqybe}
\sum R(x_1 \otimes y_1)R(x_2 \otimes z_1)R(y_2 \otimes z_2) =
\sum R(y_1 \otimes z_1)R(x_1 \otimes z_2)R(x_2 \otimes y_2)
\end{equation}
for $x,y,z \in A$.  See, e.g., \cite[Lemma 2.2.3]{majid}.  This is a consequence of \eqref{axiom2} (with $\alpha = Id$) and \eqref{axiom3}.

The following result shows that the Hom-cobraiding form in a cobraided Hom-bialgebra satisfies two Hom-type generalizations of the OQYBE.


\begin{theorem}
\label{thm:oqhybe}
Let $(A,\mu,\Delta,\alpha,R)$ be a cobraided Hom-bialgebra.  Then for all $x,y,z \in A$ we have
\begin{equation}
\label{oqhybe}
\begin{split}
\sum & R(x_1 \otimes \alpha(y_1))R(x_2 \otimes \alpha(z_1))R(y_2 \otimes z_2)\\
&= \sum R(y_1 \otimes z_1)R(x_1 \otimes \alpha(z_2))R(x_2 \otimes \alpha(y_2))
\end{split}
\end{equation}
and
\begin{equation}
\label{oqhybe2}
\begin{split}
\sum & R(x_1 \otimes y_1)R(\alpha(x_2) \otimes z_1)R(\alpha(y_2) \otimes z_2)\\
&= \sum R(\alpha(y_1) \otimes z_1)R(\alpha(x_1) \otimes z_2)R(x_2 \otimes y_2),
\end{split}
\end{equation}
called the \textbf{operator quantum Hom-Yang-Baxter equations} (OQHYBEs).
\end{theorem}

\begin{proof}
To prove \eqref{oqhybe}, we compute as follows:
\[
\begin{split}
R(x_1 \otimes \alpha(y_1))R(x_2 \otimes \alpha(z_1))R(y_2 \otimes z_2)
&= R(\alpha(x) \otimes z_1y_1)R(y_2 \otimes z_2) \quad\text{by \eqref{axiom2}}\\
&= R(\alpha(x) \otimes z_1y_1R(y_2 \otimes z_2))\\
&= R(\alpha(x) \otimes R(y_1 \otimes z_1)y_2z_2) \quad\text{by \eqref{axiom3}}\\
&= R(y_1 \otimes z_1)R(\alpha(x) \otimes y_2z_2)\\
&= R(y_1 \otimes z_1)R(x_1 \otimes \alpha(z_2))R(x_2 \otimes \alpha(y_2))\quad\text{by \eqref{axiom2}}.
\end{split}
\]
Similarly, to prove \eqref{oqhybe2}, we compute as follows:
\[
\begin{split}
R(x_1 \otimes y_1)R(\alpha(x_2) \otimes z_1)R(\alpha(y_2) \otimes z_2)
&= R(x_1 \otimes y_1)R(x_2y_2 \otimes \alpha(z)) \quad\text{by \eqref{axiom1}}\\
&= R(R(x_1 \otimes y_1)x_2y_2 \otimes \alpha(z))\\
&= R(y_1x_1R(x_2 \otimes y_2) \otimes \alpha(z)) \quad\text{by \eqref{axiom3}}\\
&= R(y_1x_1 \otimes \alpha(z))R(x_2 \otimes y_2)\\
&= R(\alpha(y_1) \otimes z_1)R(\alpha(x_1) \otimes z_2)R(x_2 \otimes y_2) \quad\text{by \eqref{axiom1}}.
\end{split}
\]
\end{proof}

The OQHYBEs \eqref{oqhybe} and \eqref{oqhybe2} are the operator forms of the QHYBEs in \cite{yau8}, which are Hom-type, non-associative generalizations of the QYBE \eqref{qybe}.  The OQHYBE \eqref{oqhybe} will play a key role in Theorem ~\ref{thm:comod} (specifically, the proof of Lemma ~\ref{lem1:comod}) when we generate solutions of the Hom-Yang-Baxter equation \cite{yau5,yau6} \eqref{hybe} from comodules of suitable cobraided Hom-bialgebras.


We end this section with some alternative characterizations of the axioms \eqref{axiom1} and \eqref{axiom2}.  Let $(A,\mu,\Delta,\alpha)$ be a Hom-bialgebra and $R$ be an arbitrary bilinear form on $A$.  Recall that $A^* = \Hom(A,\bk)$ denotes the linear dual of $A$.  Define the linear maps $\pi_1,\pi_1',\pi_2,\pi_2' \colon A \to A^*$ by
\begin{equation}
\label{pi}
\begin{split}
\pi_1(x) &= R(-\otimes \alpha(x)),\quad \pi_1'(x) = R(\alpha(-)\otimes x),\\
\pi_2(x) &= R(\alpha(x) \otimes -),\quad \pi_2'(x) = R(x \otimes \alpha(-))
\end{split}
\end{equation}
for $x \in A$.  In the following two results, we use the induced operations $\Delta^* \colon A^* \otimes A^* \to A^*$ \eqref{deltadual} and $\mu^* \colon A^* \to A^* \otimes A^*$ \eqref{mudual} (when $A$ is finite dimensional) on the dual space.

\begin{theorem}
\label{thm:axiom1}
Let $(A,\mu,\Delta,\alpha)$ be a Hom-bialgebra and $R$ be a bilinear form on $A$.  With the notations \eqref{pi}, the following statements are equivalent:
\begin{enumerate}
\item
The condition \eqref{axiom1} holds.
\item
The diagram
\begin{equation}
\label{axiom1'}
\SelectTips{cm}{10}
\xymatrix{
A \otimes A \ar[rr]^-{\pi_2 \otimes \pi_2} \ar[d]_-{\mu} & & A^* \otimes A^* \ar[d]^-{\Delta^*}\\
A \ar[rr]^-{\pi_2'} & & A^*
}
\end{equation}
is commutative.
\end{enumerate}
If $A$ is finite dimensional, then the two statements above are also equivalent to the commutativity of the diagram
\begin{equation}
\label{axiom1''}
\SelectTips{cm}{10}
\xymatrix{
A \ar[rr]^-{\pi_1} \ar[d]_-{\Delta} & & A^* \ar[d]^-{\mu^*}\\
A \otimes A \ar[rr]^-{\pi_1' \otimes \pi_1'} & & A^* \otimes A^*.
}
\end{equation}
\end{theorem}

\begin{proof}
To prove the equivalence between \eqref{axiom1} and the commutativity of the diagram \eqref{axiom1'}, note that the left-hand side in \eqref{axiom1} is
\[
R(xy \otimes \alpha(z)) = \langle \pi_2'(xy),z\rangle.
\]
On the other hand, the right-hand side in \eqref{axiom1} is
\[
\begin{split}
R(\alpha(x) \otimes z_1)R(\alpha(y) \otimes z_2)
&= \langle \pi_2(x),z_1\rangle \langle \pi_2(y),z_2\rangle\\
&= \langle \pi_2(x) \otimes \pi_2(y), \Delta(z)\rangle\\
&= \langle \Delta^*(\pi_2(x),\pi_2(y)), z\rangle.
\end{split}
\]
Therefore, \eqref{axiom1} is equivalent to the equality
\[
\langle \pi_2'(xy),z\rangle = \langle \Delta^*(\pi_2(x),\pi_2(y)), z\rangle
\]
for all $x,y,z \in A$, which in turn is equivalent to the commutativity of the diagram \eqref{axiom1'}.

For the second assertion, assume now that $A$ is finite dimensional, so $\mu^*$ is well-defined.  Then the left-hand side in \eqref{axiom1} is
\begin{equation}
\label{axiom1com}
R(xy \otimes \alpha(z)) = \langle\pi_1(z),xy\rangle = \langle \mu^*(\pi_1(z)), x \otimes y \rangle.
\end{equation}
On the other hand, the right-hand side in \eqref{axiom1} is
\begin{equation}
\label{axiom1com2}
\begin{split}
R(\alpha(x) \otimes z_1)R(\alpha(y) \otimes z_2)
&= \langle \pi_1'(z_1),x\rangle \langle \pi_1'(z_2),y\rangle\\
&= \langle \pi_1'(z_1) \otimes \pi_1'(z_2), x \otimes y\rangle\\
&= \langle \pi_1'^{\otimes 2}(\Delta(z)), x \otimes y\rangle.
\end{split}
\end{equation}
It follows from \eqref{axiom1com} and \eqref{axiom1com2} that \eqref{axiom1} is equivalent to the commutativity of the diagram \eqref{axiom1''}.
\end{proof}

The following result is the analogue of Theorem ~\ref{thm:axiom1} for the axiom \eqref{axiom2}.  Its proof is completely analogous to that of Theorem ~\ref{thm:axiom2}, so we will omit it.

\begin{theorem}
\label{thm:axiom2}
Let $(A,\mu,\Delta,\alpha)$ be a Hom-bialgebra and $R$ be a bilinear form on $A$.  With the notations \eqref{pi}, the following statements are equivalent:
\begin{enumerate}
\item
The condition \eqref{axiom2} holds.
\item
The diagram
\begin{equation}
\label{axiom2'}
\SelectTips{cm}{10}
\xymatrix{
A \otimes A \ar[rr]^-{\pi_1 \otimes \pi_1} \ar[d]_-{\mu} & & A^* \otimes A^* \ar[d]^-{\Delta^{*op}}\\
A \ar[rr]^-{\pi_1'} & & A^*
}
\end{equation}
is commutative, where $\Delta^{*op} = \Delta^* \circ \tau$.
\end{enumerate}
If $A$ is finite dimensional, then the two statements above are also equivalent to the commutativity of the diagram
\begin{equation}
\label{axiom2''}
\SelectTips{cm}{10}
\xymatrix{
A \ar[rr]^-{\pi_2} \ar[d]_-{\Delta^{op}} & & A^* \ar[d]^-{\mu^*}\\
A \otimes A \ar[rr]^-{\pi_2' \otimes \pi_2'} & & A^* \otimes A^*,
}
\end{equation}
where $\Delta^{op} = \tau \circ \Delta$.
\end{theorem}

\section{Twisting cobraided bialgebras into cobraided Hom-bialgebras}
\label{sec:twist}

In this section, we give the first of three general methods for constructing cobraided Hom-bialgebras (Definition ~\ref{def:dqt}).  We will twist the (co)multiplications in cobraided bialgebras \cite{hay,lt,majid91,sch} (Example ~\ref{ex:dqtbialgebra}) along any bialgebra endomorphisms (Theorem ~\ref{thm:twist}).  Under this twisting procedure, the bilinear form $R$ stays the same.  As a result, every cobraided bialgebra gives rise to a family (usually infinite) of cobraided Hom-bialgebras.   We illustrate this twisting procedure with several quantum groups related to quantum matrices (Examples ~\ref{ex:mq} - \ref{ex:mq'}).  Corollary ~\ref{cor:frt} and  Examples ~\ref{ex:groupbi} - ~\ref{ex:Z} give further illustrations of Theorem ~\ref{thm:twist}.

\begin{theorem}
\label{thm:twist}
Let $(A,\mu,\Delta,R)$ be a cobraided bialgebra with multiplication $\mu$ and comultiplication $\Delta$, and let $\alpha \colon A \to A$ be a bialgebra morphism.  Then
\[
A_\alpha = (A,\mu_\alpha,\Delta_\alpha,\alpha,R)
\]
is a cobraided Hom-bialgebra, where $\mu_\alpha = \alpha \circ \mu$ and $\Delta_\alpha = \Delta \circ \alpha$.
\end{theorem}

\begin{proof}
As discussed in Example ~\ref{ex:homas}, $(A,\mu_\alpha,\Delta_\alpha,\alpha)$ is a Hom-bialgebra.  It remains to check the three conditions \eqref{dqtaxioms} with multiplication $\mu_\alpha$ and comultiplication $\Delta_\alpha$.

For an element $z \in A$, we write $\Delta(z) = \sum_{(z)} z_1 \otimes z_2$ as usual.  Since $\Delta_\alpha = \Delta \circ \alpha$, we have $\Delta_\alpha(z) = \sum_{(\alpha(z))} \alpha(z)_1 \otimes \alpha(z)_2$.  We use concatenation to denote the multiplication $\mu$ in $A$, so $\mu_\alpha(x,y) = \alpha(xy) = \alpha(x)\alpha(y)$.  Since $A$ is a cobraided bialgebra, it satisfies
\begin{equation}
\label{ax1}
R(xy \otimes z) = \sum_{(z)} R(x \otimes z_1)R(y \otimes z_2).
\end{equation}
To check the axiom \eqref{axiom1}, we compute as follows:
\[
\begin{split}
R(\mu_\alpha(x,y) \otimes \alpha(z))
&= R(\alpha(x)\alpha(y) \otimes \alpha(z))\\
&= \sum_{(\alpha(z))} R(\alpha(x) \otimes \alpha(z)_1)R(\alpha(y) \otimes \alpha(z)_2)\quad\text{by \eqref{ax1}}.
\end{split}
\]
This proves \eqref{axiom1} for $A_\alpha$.

Similarly, one uses the property
\[
R(x \otimes yz) = \sum_{(x)} R(x_1 \otimes z)R(x_2 \otimes y)
\]
in $A$ to conclude that
\[
\begin{split}
R(\alpha(x) \otimes \mu_\alpha(y,z))
&= R(\alpha(x) \otimes \alpha(y)\alpha(z))\\
&= \sum_{(\alpha(x))} R(\alpha(x)_1 \otimes \alpha(z))R(\alpha(x)_2 \otimes \alpha(y)).
\end{split}
\]
This proves \eqref{axiom2} for $A_\alpha$.

Finally, to check the axiom \eqref{axiom3} in $A_\alpha$, we compute as follows:
\[
\begin{split}
\sum \mu_\alpha(\alpha(y)_1,\alpha(x)_1)R(\alpha(x)_2 \otimes \alpha(y)_2)
&= \alpha\left(\sum \alpha(y)_1\alpha(x)_1R(\alpha(x)_2 \otimes \alpha(y_2))\right)\\
&= \alpha\left(\sum R(\alpha(x)_1 \otimes \alpha(y)_1)\alpha(x)_2\alpha(y)_2\right) \quad\text{by \eqref{axiom3} in $A$}\\
&= \sum R(\alpha(x)_1 \otimes \alpha(y)_1) \mu_\alpha(\alpha(x)_2,\alpha(y)_2).
\end{split}
\]
In the above computation, we have $\sum = \sum_{(\alpha(x))(\alpha(y))}$. In the first and the third equalities, we used the linearity of $\alpha$ and the definition $\mu_\alpha = \alpha \circ \mu$.
\end{proof}

\begin{remark}
\label{rk:twist}
In the proof of Theorem ~\ref{thm:twist}, we never used the (co)unit in the bialgebra $A$.  Also, we do not need the invertibility of the bilinear form $R$ in $\Hom(A^{\otimes 2},\bk)$.
\end{remark}

\begin{remark}
\label{rk2:twist}
In Theorem ~\ref{thm:twist}, an un-twisted structure is twisted along an endomorphism into a Hom-type structure.  A twisting procedure along these lines was first employed in \cite{yau2} to construct examples of Hom-associative and Hom-Lie algebras.  It has been applied and generalized in \cite{ama,ams,ms4,yau3,yau4,yau7,yau8} to construct other Hom-type algebras.  As discussed in \cite{fg2,gohr}, this twisting procedure is also useful in understanding the structure of Hom-associative algebras.
\end{remark}

We now give a series of examples that illustrate Theorem ~\ref{thm:twist}.  For the rest of this section, we work over a field $\bk$ of characteristic $0$.

\begin{example}[\textbf{Hom-quantum matrices}]
\label{ex:mq}
This example is about a Hom version of the quantum group $\mq$ of quantum matrices.  It will play a major role in section ~\ref{sec:coaction} when we discuss Hom versions of quantum geometry.  The quantum group $\mq$ can be constructed using the FRT construction \cite{rft} and is discussed in many books on quantum groups, such as \cite[Ch.7]{cp}, \cite[Ch.10]{es}, \cite[Ch.IV and VIII]{kassel}, \cite[Ch.4]{majid}, and \cite[Ch.3]{street}.

Let us first recall the bialgebra structure of $\mq$.  Fix a scalar $q \in \bk$ such that $q^{\frac{1}{2}}$ exists and is invertible with $q^2 \not= -1$.  Let $x$ and $y$ be variables that satisfy the \emph{quantum commutation relation}:
\begin{equation}
\label{qcom}
yx = qxy.
\end{equation}
Let $a,b,c,d$ be variables, each of which commuting with both $x$ and $y$.  Define the variables $x'$, $y'$, $x''$, and $y''$ by
\begin{equation}
\label{xy'}
\begin{pmatrix}x'\\y'\end{pmatrix}
= \abcd\xymat \quad\text{and}\quad
\begin{pmatrix}x''\\y''\end{pmatrix}
= \begin{pmatrix}a & c\\b & d\end{pmatrix}\xymat.
\end{equation}
Then the two pairs of variables $(x',y')$ and $(x'',y'')$ both satisfy the quantum commutation relation \eqref{qcom} if and only if the following six relations hold:
\begin{equation}
\label{six}
\begin{split}
ab &= q^{-1}ba,\ bd = q^{-1}db,\ ac = q^{-1}ca,\ cd = q^{-1}dc,\\
bc &= cb,\ ad- da = (q^{-1} - q)bc.
\end{split}
\end{equation}
This is proved by a direct computation.  The algebra $\mq$ is defined as the quotient $\bk\{a,b,c,d\}/\text{\eqref{six}}$, i.e., the unital associative algebra generated by $a,b,c$, and $d$ with relations \eqref{six}.  Of course, this definition of the algebra $\mq$ does not require the quantum commutation relation.  However, the relations \eqref{six} become more intuitive and geometric when considered with the quantum commutation relation for the variables $(x',y')$ and $(x'',y'')$ in \eqref{xy'}.

The algebra $\mq$ becomes a bialgebra when equipped with the comultiplication $\Delta$ determined by
\begin{equation}
\label{deltaMq}
\Delta\abcd = \abcd \otimes \abcd,
\end{equation}
i.e.,
\begin{equation}
\label{deltamq}
\begin{split}
\Delta(a) &= a \otimes a + b \otimes c, \quad
\Delta(b) = a \otimes b + b \otimes d,\\
\Delta(c) &= c \otimes a + d \otimes c, \quad
\Delta(d) = c \otimes b + d \otimes d.
\end{split}
\end{equation}
Note that the bialgebra $\mq$ is both non-commutative (for $q \not= 1$) and non-cocommutative.  The matrix notation \eqref{deltaMq} for morphisms will be used several more times in the rest of this paper.

The bialgebra $\mq$ becomes a cobraided bialgebra (Example ~\ref{ex:dqtbialgebra}) when equipped with the cobraiding form $R \colon \mq \otimes \mq \to \bk$ determined by the so-called \emph{$sl_2$ $R$-matrix}:
\begin{equation}
\label{mqR}
\mqr = q^{-\frac{1}{2}}\mr.
\end{equation}
The bilinear form $R$ is extended to all of $\mq^{\otimes 2}$ using
\begin{equation}
\label{mqR2}
\begin{split}
R(1 \otimes a) &= R(a \otimes 1) = R(1 \otimes d) = R(d \otimes 1) = 1,\\
R(1 \otimes b) &= R(b \otimes 1) = R(1 \otimes c) = R(c \otimes 1) = 0,
\end{split}
\end{equation}
and the axioms \eqref{axiom1} and \eqref{axiom2} (with $\alpha = Id$).

We now twist the cobraided bialgebra $\mq$ into an infinite family of cobraided Hom-bialgebras using Theorem ~\ref{thm:twist}.  Let $\lambda$ be any invertible  scalar in $\bk$.  Define the map $\alpha = \alpha_\lambda \colon \mq \to \mq$ by
\begin{equation}
\label{alphamq}
\alpha\abcd =
\begin{pmatrix}
\alpha(a) & \alpha(b)\\
\alpha(c) & \alpha(d)
\end{pmatrix} =
\begin{pmatrix}
a & \lambda b\\
\lambda^{-1}c & d
\end{pmatrix},
\end{equation}
extended multiplicatively to all of $\mq$.  The map $\alpha$ is indeed a well-defined algebra morphism because it preserves the six relations in \eqref{six}.  Moreover, it satisfies $\alpha^{\otimes 2} \circ \Delta = \Delta \circ \alpha$ by \eqref{deltamq}.  Thus, $\alpha$ is a bialgebra morphism on $\mq$.  It is invertible with inverse $\alpha_{\lambda^{-1}}$.  By Theorem ~\ref{thm:twist} we have a cobraided Hom-bialgebra
\begin{equation}
\label{mqalpha}
\mqalpha = (\mq,\mu_\alpha = \alpha \circ \mu,\Delta_\alpha = \Delta \circ \alpha,\alpha,R),
\end{equation}
where $\mu$ denotes the multiplication in $\mq$.  Using \eqref{deltaMq} and \eqref{alphamq} the twisted comultiplication $\Delta_\alpha = \alpha^{\otimes 2} \circ \Delta$ is given on the generators by
\begin{equation}
\label{deltaalphamq}
\Delta_\alpha\abcd = \abcdlambda \otimes \abcdlambda.
\end{equation}
It follows from \eqref{mqR}, \eqref{mqR2}, and \eqref{alphamq} that the bilinear form $R$ is $\alpha$-invariant, i.e, $R = R \circ \alpha^{\otimes 2}$.

Letting $\lambda$ run through the invertible scalars in $\bk$, we thus have an infinite, $1$-parameter family $\mqalpha$ of cobraided Hom-bialgebra twistings of the quantum group $\mq$.  Since $\mq$ is non-commutative (provided $q \not= 1$) and non-cocommutative, so is each cobraided Hom-bialgebra $\mqalpha$, which is also non-associative and non-coassociative.

In Example ~\ref{ex2:mq} below, we will consider the cobraided Hom-bialgebras $\mqalpha$ in a more general context.
\qed
\end{example}

The following two examples are both related to the quantum group $\mq$ of Example ~\ref{ex:mq}.  The general references are still \cite[Ch.7]{cp}, \cite[Ch.10]{es}, \cite[Ch.IV and VIII]{kassel}, \cite[Ch.4]{majid}, \cite{rft}, and \cite[Ch.3]{street}.

\begin{example}[\textbf{Hom-quantum general linear groups}]
\label{ex:gl}
With the same notations as in Example ~\ref{ex:mq}, define the \emph{quantum determinant}
\begin{equation}
\label{detq}
\detq = ad - q^{-1}bc.
\end{equation}
It follows from the relations \eqref{six} and \eqref{deltamq} that the quantum determinant lies in the center of $\mq$ and is group-like, i.e., $\Delta(\detq) = \detq \otimes \detq$.  The quantum general linear group is obtained from $\mq$ by adjoining an inverse to the quantum determinant.  More precisely, let $t$ be a variable that commutes with each of $a,b,c$, and $d$.  The \textbf{quantum general linear group} is defined as the algebra
\[
\gl = \mq[t]/(t\detq - 1).
\]
The formulas \eqref{deltamq}, \eqref{mqR}, and \eqref{mqR2} give $\gl$ the structure of a cobraided bialgebra (Example ~\ref{ex:dqtbialgebra}).

For any invertible scalar $\lambda \in \bk$, the map $\alpha = \alpha_\lambda$ defined in \eqref{alphamq} extends to a bialgebra automorphism on $\gl$ if we define $\alpha(t) = t$.  Indeed, we only need to see that $\alpha$ preserves the relation $t\detq = 1$.  This is true because $\alpha(1) = 1$, $\alpha(t) = t$, and
\begin{equation}
\label{alphadetq}
\begin{split}
\alpha(\detq)
&= \alpha(a)\alpha(d) - q^{-1}\alpha(b)\alpha(c)\\
&= ad - q^{-1}(\lambda b)(\lambda^{-1}c)\\
&= \detq.
\end{split}
\end{equation}
Thus, by Theorem ~\ref{thm:twist} we have a cobraided Hom-bialgebra
\begin{equation}
\label{glalpha}
\glalpha = (\gl, \mu_\alpha=\alpha\circ\mu, \Delta_\alpha=\Delta \circ \alpha, \alpha,R)
\end{equation}
with an $\alpha$-invariant Hom-cobraiding form $R$ \eqref{mqR}.

Letting $\lambda$ run through the invertible scalars in $\bk$, we thus have an infinite, $1$-parameter family $\glalpha$ of cobraided Hom-bialgebra twistings of the quantum general linear group $\gl$.  Since $\gl$ is non-commutative (provided $q \not= 1$) and non-cocommutative, so is each cobraided Hom-bialgebra $\glalpha$, which is also non-associative and non-coassociative.
\qed
\end{example}

\begin{example}[\textbf{Hom-quantum special linear groups}]
\label{ex:sl}
The \textbf{quantum special linear group} is the quotient algebra
\[
\slq = \mq/(\detq - 1),
\]
where $\detq$ is as in \eqref{detq}.  The formulas  \eqref{deltamq}, \eqref{mqR}, and \eqref{mqR2} give $\slq$ the structure of a cobraided bialgebra (Example ~\ref{ex:dqtbialgebra}).  For each invertible scalar $\lambda \in \bk$, the map $\alpha = \alpha_\lambda$ \eqref{alphamq} induces a bialgebra automorphism on $\slq$ because of \eqref{alphadetq}.  Thus, by Theorem ~\ref{thm:twist} we have a cobraided Hom-bialgebra
\begin{equation}
\label{slalpha}
\slqalpha = (\slq, \mu_\alpha=\alpha\circ\mu, \Delta_\alpha=\Delta \circ \alpha, \alpha,R)
\end{equation}
with an $\alpha$-invariant Hom-cobraiding form $R$ \eqref{mqR}.

Letting $\lambda$ run through the invertible scalars in $\bk$, we thus have an infinite, $1$-parameter family $\slqalpha$ of cobraided Hom-bialgebra twistings of the quantum special linear group $\slq$.  Since $\slq$ is non-commutative (provided $q \not= 1$) and non-cocommutative, so is each cobraided Hom-bialgebra $\slqalpha$, which is also non-associative and non-coassociative.\qed
\end{example}

There are interesting variations of the quantum group $\mq$ in Example ~\ref{ex:mq}.  In the next two examples, we demonstrate that our twisting procedure (Theorem ~\ref{thm:twist}) applies just as well to these non-standard quantum groups.

\begin{example}[\textbf{Multi-parameter Hom-quantum matrices}]
\label{ex:mpq}
This example is about a $2$-parameter version $\mpq$ of $\mq$ discussed in Example ~\ref{ex:mq}.  The references are \cite{dmmz,rft} and \cite[Example 4.2.12]{majid}.  Let $p$ and $q$ be invertible scalars in the ground field $\bk$.  For independent variables $a$, $b$, $c$, and $d$, consider the six relations:
\begin{equation}
\label{sixrel}
\begin{split}
ab &= q^{-1}ba,\ bd = p^{-1}db,\ ac = p^{-1}ca,\ cd = q^{-1}dc,\\
bc &= qp^{-1}cb,\ ad - da = (q^{-1} - p)bc.
\end{split}
\end{equation}
Then the algebra $\mpq$ is defined as the quotient $\bk\{a,b,c,d\}/\eqref{sixrel}$.  It becomes a bialgebra when equipped with the comultiplication $\Delta$ given in \eqref{deltaMq}.  In particular, if $p = q$, then $\mpq$ is the bialgebra $\mq$.  The bialgebra $\mpq$ becomes a cobraided bialgebra when equipped with the bilinear form $R \colon \mpq \otimes \mpq \to \bk$ determined by
\begin{equation}
\label{mpqR}
\mqr = \mpqr.
\end{equation}
It is extended to all of $\mpq^{\otimes 2}$ by \eqref{mqR2} and the axioms \eqref{axiom1} and \eqref{axiom2} (with $\alpha = Id$).

Given any invertible scalar $\lambda \in \bk$, there is a bialgebra morphism $\alpha = \alpha_\lambda \colon \mpq \to \mpq$ determined by \eqref{alphamq}.  It is well-defined because it preserves the six relations \eqref{sixrel}, and it is compatible with the comultiplication $\Delta$.  By Theorem ~\ref{thm:twist} we have a cobraided Hom-bialgebra
\begin{equation}
\label{mpqalpha}
\mpq_\alpha = (\mpq,\mu_\alpha = \alpha \circ \mu,\Delta_\alpha = \Delta \circ \alpha,\alpha,R),
\end{equation}
in which the twisted comultiplication $\Delta_\alpha$ is given on the generators as in \eqref{deltaalphamq}.  The bilinear form $R$ \eqref{mpqR} is $\alpha$-invariant.

Letting $\lambda$ run through the invertible scalars in $\bk$, we thus have an infinite, $1$-parameter family $\mpq_\alpha$ of cobraided Hom-bialgebra twistings of the $2$-parameter quantum group $\mpq$.  Since $\mpq$ is non-commutative (provided $p \not= 1$ or $q \not= 1$) and non-cocommutative, so is each cobraided Hom-bialgebra $\mpq_\alpha$, which is also non-associative and non-coassociative.
\qed
\end{example}

\begin{example}[\textbf{Non-standard Hom-quantum matrices}]
\label{ex:mq'}
This example is about a non-standard variation $\mqns$ of $\mq$.  The references are \cite{gjw,rft} and \cite[Example 4.2.13]{majid}.  Let $q$ be an invertible scalar in $\bk$ with $q^2 \not= -1$.  Consider the following relations:
\begin{equation}
\label{mq'rel}
\begin{split}
ab &= q^{-1}ba,\ bd = -qdb,\ ac = q^{-1}ca,\ cd = -qdc,\\
b^2 &= c^2 = 0,\ bc = cb,\ ad - da = (q^{-1}-q)bc.
\end{split}
\end{equation}
Then the algebra $\mqns$ is defined as the quotient $\bk\{a,b,c,d\}/\eqref{mq'rel}$.  It becomes a bialgebra when equipped with the comultiplication $\Delta$ in \eqref{deltaMq}.  Moreover, $\mqns$ becomes a cobraided bialgebra when equipped with the bilinear form $R \colon \mqns \otimes \mqns \to \bk$ determined by the so-called \emph{Alexander-Conway $R$-matrix}:
\begin{equation}
\label{mq'R}
\mqr = \mqnsmatrix
\end{equation}
It is extended to all of $\mqns^{\otimes 2}$ by \eqref{mqR2} and the axioms \eqref{axiom1} and \eqref{axiom2} (with $\alpha = Id$).

Given any invertible scalar $\lambda \in \bk$, there is a bialgebra morphism $\alpha = \alpha_\lambda \colon \mqns \to \mqns$ determined by \eqref{alphamq}.  It is well-defined because it preserves all the relations in \eqref{mq'rel}, and it is compatible with the comultiplication $\Delta$.  By Theorem ~\ref{thm:twist} we have a cobraided Hom-bialgebra
\begin{equation}
\label{mq'alpha}
\mqns_\alpha = (\mqns,\mu_\alpha = \alpha \circ \mu,\Delta_\alpha = \Delta \circ \alpha,\alpha,R),
\end{equation}
in which the twisted comultiplication $\Delta_\alpha$ is given on the generators as in \eqref{deltaalphamq}.  The bilinear form $R$ \eqref{mq'R} is $\alpha$-invariant.

Letting $\lambda$ run through the invertible scalars in $\bk$, we thus have an infinite, $1$-parameter family $\mqns_\alpha$ of cobraided Hom-bialgebra twistings of the non-standard quantum group $\mqns$.  Since $\mqns$ is non-commutative and non-cocommutative, so is each cobraided Hom-bialgebra $\mqns_\alpha$, which is also non-associative and non-coassociative.

We will revisit this example in section ~\ref{sec:coaction} when we discuss Hom versions of quantum geometry.
\qed
\end{example}

\section{FRT Hom-quantum groups}
\label{sec:frt}

In this section, we show that the twisting procedure in Theorem ~\ref{thm:twist} applies naturally to a large class of quantum groups, namely, the FRT quantum groups.  We will construct explicit bialgebra morphisms on each  FRT quantum group (Theorem ~\ref{thm:frt}), so Theorem ~\ref{thm:twist} can then be applied (Corollary ~\ref{cor:frt}).

Let us first recall some relevant definitions.  Throughout this section, we work over a field $\bk$ of characteristic $0$.


\begin{definition}
Let $V$ be a vector space and $\gamma \colon V^{\otimes 2} \to V^{\otimes 2}$ be a linear automorphism.  We say that $\gamma$ is an \textbf{$R$-matrix} if it satisfies the \textbf{Yang-Baxter equation} (YBE)
\begin{equation}
\label{ybe}
(Id \otimes \gamma) \circ (\gamma \otimes Id) \circ (Id \otimes \gamma) = (\gamma \otimes Id) \circ (Id \otimes \gamma) \circ (\gamma \otimes Id).
\end{equation}
\end{definition}

The YBE was introduced in the work of McGuire, Yang \cite{yang}, and Baxter \cite{baxter,baxter2} in statistical mechanics.  Drinfel'd's braided bialgebras \cite{dri87} and the dual objects of cobraided bialgebras (Example ~\ref{ex:dqtbialgebra}) were introduced partly to construct $R$-matrices in their (co)modules.  A Hom-type analogue of the YBE, called the Hom-Yang-Baxter equation (HYBE), was studied in \cite{yau5,yau6,yau8}.  Solutions of the HYBE associated to comodules of cobraided Hom-bialgebras will be discussed in section ~\ref{sec:hybe}.  A Hom version of the classical YBE (CYBE) \cite{dri87,skl1,skl2} was studied in \cite{yau7}.  The reader may consult \cite{perk} for discussions of the YBE, the CYBE, the QYBE \eqref{qybe}, and their uses in physics.

Suppose that $(A,R)$ is a cobraided bialgebra and that $V$ is an $A$-comodule with structure map $\rho$.  Write $\rho(v) = \sum v_A \otimes v_V$ for $v \in V$.  Then there is a canonically associated $R$-matrix $B_{V,V} \colon V^{\otimes 2} \to V^{\otimes 2}$ defined as
\begin{equation}
\label{BVV}
B_{V,V}(v \otimes w) = \sum R(w_A \otimes v_A) w_V \otimes v_V.
\end{equation}
The FRT construction \cite{rft} goes in the other direction.  It associates to each $R$-matrix on a finite dimensional vector space $V$ a cobraided bialgebra that coacts non-trivially on $V$, from which the $R$-matrix can be recovered as $B_{V,V}$ \eqref{BVV}.  Let us now recall the details of the FRT construction from \cite{rft}.  Expositions on the FRT construction can be found in, e.g, \cite[VIII.6]{kassel} and \cite[Ch.4.1]{majid}.


Let $V$ be a finite dimensional vector space with a basis $\{v_1, \ldots ,v_N\}$.  Let $\gamma \colon V^{\otimes 2} \to V^{\otimes 2}$ be an $R$-matrix on $V$ whose structure constants $\{c_{ij}^{mn}\}$ are determined by
\begin{equation}
\label{gamma}
\gamma(v_i \otimes v_j) = \sum_{m,n=1}^N c_{ij}^{mn} v_m \otimes v_n.
\end{equation}
Let $\{T_i^j\}_{i,j=1}^N$ be $N^2$ independent variables.

\begin{definition}[\textbf{FRT quantum groups}]
\label{def:frt}
Define the $\bk$-algebra
\begin{equation}
\label{agamma}
\agamma = \bk\{T_i^j \colon 1 \leq i,j \leq N\}/(C_{ij}^{mn}=0),
\end{equation}
where the relations (for $i,j,m,n \in \{1,\ldots,N\}$) are defined as
\begin{equation}
\label{frtrelations}
C_{ij}^{mn} = \sum_{k,l=1}^N c_{ij}^{kl}T_k^mT_l^n - \sum_{k,l=1}^N T_i^kT_j^lc_{kl}^{mn} = 0.
\end{equation}
The algebra $\agamma$ becomes a bialgebra when equipped with the comultiplication
\begin{equation}
\label{frtdelta}
\Delta(T_i^j) = \sum_{k=1}^N T_i^k \otimes T_k^j.
\end{equation}
The bialgebra $\agamma$ becomes a cobraided bialgebra when equipped with the cobraiding form $R \colon \agamma^{\otimes 2} \to \bk$ determined by
\begin{equation}
\label{frtR}
\begin{split}
R(T_i^m \otimes T_j^n) &= c_{ji}^{mn},\\
R(1 \otimes T_i^j) = R(T_i^j \otimes 1) &= \delta_{ij}.
\end{split}
\end{equation}
It is extended to all of $\agamma^{\otimes 2}$ using \eqref{axiom1} and \eqref{axiom2} (with $\alpha = Id$).  The cobraided bialgebra $\agamma$ is called the \textbf{FRT quantum group} associated to $\gamma$.
\end{definition}

The $R$-matrix $\gamma$ can be recovered from the $\agamma$-comodule structure $\rho \colon V \to \agamma \otimes V$ on $V$ defined by
\begin{equation}
\label{frtcomodule}
\rho(v_i) = \sum_{k=1}^N T_i^k \otimes v_k.
\end{equation}
Indeed, with this $\agamma$-comodule structure, it follows from \eqref{gamma} and \eqref{frtR} that $\gamma = B_{V,V}$ \eqref{BVV}.  Therefore, every $R$-matrix on a finite dimensional vector space $V$ must be of the form $B_{V,V}$ for some cobraided bialgebra coacting on $V$.

We would like to apply the twisting procedure in Theorem ~\ref{thm:twist} to the FRT quantum groups.  In the following result, we construct some explicit bialgebra endomorphisms on $\agamma$, for which Theorem ~\ref{thm:twist} can then be applied.

\begin{theorem}
\label{thm:frt}
Let $\agamma$ be the FRT quantum group associated to the $R$-matrix $\gamma \colon V^{\otimes 2} \to V^{\otimes 2}$ \eqref{gamma}.  Let $\lambda_i \in \bk$ be invertible scalars such that
\begin{equation}
\label{lambdac}
\lambda_i \lambda_j c_{ij}^{mn} = \lambda_m\lambda_n c_{ij}^{mn}
\end{equation}
for all $i,j,m,n \in \{1,\ldots,N\}$.  Then the following statements hold.
\begin{enumerate}
\item
There is a bialgebra morphism $\alpha_A \colon \agamma \to \agamma$ determined by
\begin{equation}
\label{frtalpha}
\alpha_A(T_i^j) = \lambda_i\lambda_j^{-1}T_i^j
\end{equation}
for all $i$ and $j$.
\item
The bilinear form $R$ \eqref{frtR} on $\agamma$ is $\alpha$-invariant, i.e., $R = R \circ \alpha_A^{\otimes 2}$.
\item
Consider the linear map $\alpha_V \colon V \to V$ defined as
\[
\alpha_V(v_i) = \lambda_iv_i
\]
for all $i$.  Then we have
\begin{equation}
\label{frtrhoalpha}
\begin{split}
(\alpha_A \otimes \alpha_V) \circ \rho &= \rho \circ \alpha_V,\\
\alpha_V^{\otimes 2} \circ \gamma &= \gamma \circ \alpha_V^{\otimes 2},
\end{split}
\end{equation}
where $\rho \colon V \to \agamma \otimes V$ is the $\agamma$-comodule structure map in \eqref{frtcomodule}.
\end{enumerate}
\end{theorem}

\begin{proof}
Consider the first assertion.  To show that $\alpha_A$ is an algebra morphism, we need to prove that $\alpha_A$ preserves the relations \eqref{frtrelations}.  This is proved by the following computation, where $\sum = \sum_{k,l=1}^N$:
\[
\begin{split}
\alpha_A\left(\sum c_{ij}^{kl}T_k^mT_l^n\right)
&= \sum c_{ij}^{kl}\lambda_k\lambda_l\lambda_m^{-1}\lambda_n^{-1}
T_k^mT_l^n\\
&= \lambda_i\lambda_j\lambda_m^{-1}\lambda_n^{-1}\sum c_{ij}^{kl}
T_k^mT_l^n\quad\text{by \eqref{lambdac}}\\
&= \lambda_i\lambda_j\lambda_m^{-1}\lambda_n^{-1} \sum T_i^kT_j^lc_{kl}^{mn}\quad\text{by \eqref{frtrelations}}\\
&= \sum \lambda_i\lambda_j\lambda_k^{-1}\lambda_l^{-1}T_i^kT_j^lc_{kl}^{mn}
\quad\text{by \eqref{lambdac}}\\
&= \alpha\left(\sum T_i^kT_j^lc_{kl}^{mn}\right).
\end{split}
\]
Since $\Delta$ \eqref{frtdelta} is also an algebra morphism, to show that $\alpha_A$ is compatible with $\Delta$ it suffices to consider the algebra generators $T_i^j$.  With $\sum = \sum_{k=1}^N$, we have
\[
\begin{split}
\Delta(\alpha_A(T_i^j))
&= \sum \lambda_i\lambda_j^{-1} T_i^k \otimes T_k^j\\
&= \sum \left(\lambda_i\lambda_k^{-1}T_i^k\right) \otimes \left(\lambda_k\lambda_j^{-1}T_k^j\right)\\
&= \alpha_A^{\otimes 2}(\Delta(T_i^j)).
\end{split}
\]
This shows that $\alpha_A$ is a bialgebra morphism.

Next we prove $R = R \circ \alpha_A^{\otimes 2}$.  Since $\alpha_A$ is a bialgebra morphism, it suffices to consider only $T_i^m \otimes T_j^n$, $1 \otimes T_i^j$, and $T_i^j \otimes 1$.  We have
\[
\begin{split}
R(\alpha_A(T_i^m) \otimes \alpha_A(T_j^n))
&= \lambda_i\lambda_j\lambda_m^{-1}\lambda_n^{-1}c_{ji}^{mn}\quad\text{by \eqref{frtR}}\\
&= c_{ji}^{mn} \quad\text{by \eqref{lambdac}}\\
&= R(T_i^m \otimes T_j^n).
\end{split}
\]
Likewise, we have
\[
\begin{split}
R(1 \otimes \alpha_A(T_i^j))
&= \lambda_i\lambda_j^{-1}\delta_{ij}\quad\text{by \eqref{frtR}}\\
&= \delta_{ij}\\
&= R(1 \otimes T_i^j).
\end{split}
\]
Similarly, we have $R(\alpha_A(T_i^j) \otimes 1) = \delta_{ij} = R(T_i^j \otimes 1)$.  This shows that $R$ \eqref{frtR} is $\alpha$-invariant.

For the last assertion, consider the first equality in \eqref{frtrhoalpha}.  We have
\[
\begin{split}
(\alpha_A \otimes \alpha_V)(\rho(v_i))
&= \sum \alpha_A(T_i^k) \otimes \alpha_V(v_k)\quad\text{by \eqref{frtcomodule}}\\
&= \sum (\lambda_i\lambda_k^{-1} T_i^k)\otimes (\lambda_k v_k)\\
&= \lambda_i\sum T_i^k \otimes v_i\\
&= \rho(\alpha_V(v_i)),
\end{split}
\]
proving the first equality in \eqref{frtrhoalpha}.  For the second equality in \eqref{frtrhoalpha}, we have
\[
\begin{split}
\gamma(\alpha_V(v_i) \otimes \alpha_V(v_j))
&= \sum \lambda_i\lambda_j c_{ij}^{mn} v_m \otimes v_n\quad\text{by \eqref{gamma}}\\
&= \sum \lambda_m\lambda_n c_{ij}^{mn} v_m \otimes v_n\quad\text{by \eqref{lambdac}}\\
&= \alpha_V^{\otimes 2}\left(\sum c_{ij}^{mn}v_m \otimes v_n\right)\\
&= \alpha_V^{\otimes 2}(\gamma(v_i \otimes v_j)).
\end{split}
\]
This proves that $\alpha_V^{\otimes 2} \circ \gamma = \gamma \circ \alpha_V^{\otimes 2}$.
\end{proof}

Applying Theorem ~\ref{thm:twist} to the cobraided bialgebra $\agamma$ and the bialgebra morphism $\alpha_A$ \eqref{frtalpha}, we obtain the following result.

\begin{corollary}
\label{cor:frt}
With the same hypotheses as in Theorem ~\ref{thm:frt}, there is a cobraided Hom-bialgebra
\[
\agammaalpha = (\agamma, \mu_\alpha = \alpha_A\circ\mu, \Delta_\alpha = \Delta \circ \alpha_A,\alpha_A,R).
\]
The Hom-cobraiding form $R$ \eqref{frtR} is $\alpha$-invariant.  The twisted comultiplication is given by
\[
\Delta_\alpha(T_i^j) = \lambda_i\lambda_j^{-1}\Delta(T_i^j) = \lambda_i\lambda_j^{-1}\sum_{k=1}^N T_i^k \otimes T_k^j.
\]
\end{corollary}

We call $\agammaalpha$ an \textbf{FRT Hom-quantum group}.  We now observe that the cobraided Hom-bialgebras $\mqalpha$ \eqref{mqalpha}, $\mpq_\alpha$ \eqref{mpqalpha}, and $\mqnsalpha$ \eqref{mq'alpha} are all FRT Hom-quantum groups.

\begin{example}[\textbf{Hom-quantum matrices as FRT Hom-quantum groups}]
\label{ex2:mq}
The quantum group $\mq = \bk\{a,b,c,d\}/\eqref{six}$ in Example ~\ref{ex:mq} is actually the FRT quantum group associated to the $sl_2$ $R$-matrix $\gamma$ in \eqref{mqR}.  Here the $sl_2$ $R$-matrix is considered as a map $\gamma \colon V^{\otimes 2} \to V^{\otimes 2}$, where $V$ is $2$-dimensional with a basis $\{v_1,v_2\}$.  The vector space $V^{\otimes 2}$ is given the basis $\{v_1 \otimes v_1, v_1 \otimes v_2, v_2 \otimes v_1, v_2 \otimes v_2\}$.  The four algebra generators $T_i^j$ for $\agamma$ are identified with those for $\mq$ by
\begin{equation}
\label{id}
\Tmatrix = \abcd.
\end{equation}
To see that the relations for these two sets of generators also coincide, note that most of the structure constants $c_{ij}^{mn}$ for $\gamma$ \eqref{gamma} are $0$. The only non-zero $c_{ij}^{mn}$ are $c_{11}^{11} = c_{22}^{22} = q^{1/2}$, $c_{12}^{21} = c_{21}^{12} = q^{-1/2}$, and $c_{21}^{21} = q^{-1/2}(q - q^{-1})$.  Using these structure constants, a direct computation shows that most of the $16$ relations $C_{ij}^{mn} = 0$ \eqref{frtrelations} for $\agamma$ are redundant, except for $C_{11}^{12}$, $C_{12}^{11}$, $C_{12}^{12}$, $C_{12}^{21}$, $C_{12}^{22}$, and $C_{22}^{12}$.  With the identification \eqref{id}, these six relations are equivalent to the relations \eqref{six} for $\mq$.  The comultiplication for $\mq$ \eqref{deltaMq} is clearly equal to that for $\agamma$ \eqref{frtdelta}.  The cobraiding form $R$ \eqref{frtR} on $\agamma$ is equal to that on $\mq$ \eqref{mqR}.   See, e.g., \cite[VIII.7]{kassel} or \cite[Example 4.2.5]{majid} for the details of the identification between $\mq$ and $\agamma$.

Let $\lambda$ be an invertible scalar in $\bk$, and set $\lambda_1 = \lambda$ and $\lambda_2 = 1$.  We claim that \eqref{lambdac} holds, i.e., $\lambda_i\lambda_jc_{ij}^{mn} = \lambda_m\lambda_nc_{ij}^{mn}$.  As noted in the previous paragraph, in the $sl_2$ $R$-matrix $\gamma$ \eqref{mqR} there are only five non-zero $c_{ij}^{mn}$, corresponding to $(ijmn) = (1111)$, $(1221)$, $(2112)$, $(2121)$, and $(2222)$.  After canceling the non-zero $c_{ij}^{mn}$ from both sides, the corresponding conditions \eqref{lambdac} are: $\lambda_1\lambda_1 = \lambda_1\lambda_1$, $\lambda_1\lambda_2 = \lambda_2\lambda_1$, $\lambda_2\lambda_1 = \lambda_1\lambda_2$, $\lambda_2\lambda_1 = \lambda_2\lambda_1$, and  $\lambda_2\lambda_2 = \lambda_2\lambda_2$. They are all trivially true.  The map $\alpha_A \colon \agamma \to \agamma$ \eqref{frtalpha} takes the form
\[
\alpha_A\Tmatrix = \Tmatrixlambda = \abcdlambda,
\]
which coincides with $\alpha_\lambda \colon \mq \to \mq$ \eqref{alphamq}. Therefore, if we apply Corollary ~\ref{cor:frt} to $\agamma = \mq$ and $(\lambda_1,\lambda_2) = (\lambda,1)$, the resulting cobraided Hom-bialgebra $\agammaalpha$ is exactly $\mqalpha$ as obtained in Example ~\ref{ex:mq}.

Similarly, the cobraided Hom-bialgebras $\mpq_\alpha$ \eqref{mpqalpha} and $\mqnsalpha$ \eqref{mq'alpha} are also FRT Hom-quantum groups, where the $R$-matrices are \eqref{mpqR} and \eqref{mq'R}, respectively.  The proof that $\mpq$ and $\mqns$ are the FRT quantum groups associated to these $R$-matrices is similar to the $\mq$ case and can be found in, e.g., \cite[Examples 4.2.12 and 4.2.13]{majid}.  The proof that, with $(\lambda_1,\lambda_2) = (\lambda,1)$, Corollary ~\ref{cor:frt} gives rise to $\mpq_\alpha$ and $\mqnsalpha$ is exactly the same as in the previous paragraph.
\qed
\end{example}

\section{Cobraided Hom-bialgebras with injective twisting maps}
\label{sec:injective}

In this section, we give the second general method for constructing cobraided Hom-bialgebras (Definition ~\ref{def:dqt}).  In Theorem ~\ref{thm:twist} we twist the (co)multiplication in a cobraided bialgebra (Example ~\ref{ex:dqtbialgebra}) along a bialgebra endomorphism, keeping the bilinear form $R$ the same.  Here we start with a cobraided Hom-bialgebra and twist its Hom-cobraiding form $R$ by $\alpha^n$ when $\alpha$ is injective (Theorem ~\ref{thm:twistinj}), keeping the (co)multiplication the same.  The result is a sequence of cobraided Hom-bialgebras, one for each $n \geq 1$, with the same (co)multiplication and twisting map as the original cobraided Hom-bialgebra.  Applied to cobraided bialgebras and injective bialgebra endomorphisms, these two twisting procedures allow one to twist both the (co)multiplication and the bilinear form $R$ simultaneously (Corollary ~\ref{cor:twist}).  We illustrate this twisting procedure with quantum group bialgebras (Example ~\ref{ex:groupbi}) and the (integral) anyon-generating quantum groups (Examples ~\ref{ex:anyon} and ~\ref{ex:Z}).


\begin{theorem}
\label{thm:twistinj}
Let $(A,\mu,\Delta,\alpha,R)$ be a cobraided Hom-bialgebra with $\alpha$ injective.  Then
\[
\an = (A,\mu,\Delta,\alpha,\rn)
\]
is also a cobraided Hom-bialgebra for each $n \geq 1$, where $\rn = R \circ (\alpha^n \otimes \alpha^n)$.
\end{theorem}

\begin{proof}
By induction it suffices to consider the case $n = 1$.  We must check the three conditions \eqref{dqtaxioms} for $\aone$, i.e., with $\ralpha = R \circ \alpha^{\otimes 2}$ instead of $R$.

First we check \eqref{axiom1} for $\aone$.  Let $x$, $y$, and $z$ be elements in $A$.  Using the (co)multiplicativity of $\alpha$ and \eqref{axiom1} for $A$, we compute as follows:
\[
\begin{split}
\ralpha(xy \otimes \alpha(z))
&= R(\alpha(xy) \otimes \alpha^2(z))\\
&= R(\alpha(x)\alpha(y) \otimes \alpha^2(z))\\
&= R(\alpha^2(x) \otimes \alpha(z)_1)R(\alpha^2(y) \otimes \alpha(z)_2)\\
&= R(\alpha^2(x) \otimes \alpha(z_1))R(\alpha^2(y) \otimes \alpha(z_2))\\
&= \ralpha(\alpha(x)\otimes z_1)\ralpha(\alpha(y)\otimes z_2).
\end{split}
\]
This proves \eqref{axiom1} for $\aone$.  The condition \eqref{axiom2} for $\aone$ is checked by essentially the same computation, using \eqref{axiom2} for $A$ instead.  Notice that the injectivity assumption of $\alpha$ has not been used yet.

The condition \eqref{axiom3} for $\aone$ means the equality
\begin{equation}
\label{ax3a1}
\sum y_1x_1 \ralpha(x_2 \otimes y_2) = \sum \ralpha(x_1 \otimes y_1)x_2y_2
\end{equation}
in $A$.  Since $\alpha$ is injective, it follows that \eqref{ax3a1} holds if and only if the two sides are equal after applying $\alpha$.  Using the (co)multiplicativity and linearity of $\alpha$ and \eqref{axiom3} of $A$, we compute as follows, where the obvious summation signs are omitted:
\[
\begin{split}
\alpha\left(y_1x_1 \ralpha(x_2 \otimes y_2)\right)
&= \alpha(y_1)\alpha(x_1)R(\alpha(x_2)\otimes \alpha(y_2))\\
&= \alpha(y)_1\alpha(x)_1R(\alpha(x)_2 \otimes \alpha(y)_2)\\
&= R(\alpha(x)_1 \otimes \alpha(y)_1)\alpha(x)_2\alpha(y)_2\\
&= R(\alpha(x_1) \otimes \alpha(y_1))\alpha(x_2)\alpha(y_2)\\
&= \ralpha(x_1 \otimes y_1)\alpha(x_2y_2)\\
&= \alpha\left(\ralpha(x_1 \otimes y_1)x_2y_2\right).
\end{split}
\]
This proves that the images under $\alpha$ of the two sides in \eqref{ax3a1} are equal, as desired.
\end{proof}

Combining Theorems ~\ref{thm:twist} and ~\ref{thm:twistinj}, we obtain the following result.

\begin{corollary}
\label{cor:twist}
Let $(A,\mu,\Delta,R)$ be a cobraided bialgebra and $\alpha \colon A \to A$ be an injective bialgebra morphism.  Then
\[
A_\alpha^{(n)} = (A,\mu_\alpha,\Delta_\alpha,\alpha,\rn)
\]
is a cobraided Hom-bialgebra for each $n \geq 1$, where $\mu_\alpha = \alpha \circ \mu$, $\Delta_\alpha = \Delta \circ \alpha$, and $\rn = R \circ (\alpha^n \otimes \alpha^n)$.
\end{corollary}

We now give some examples illustrating Theorem ~\ref{thm:twist} and Corollary ~\ref{cor:twist}.


\begin{example}[\textbf{Hom-quantum group bialgebras}]
\label{ex:groupbi}
In this example, we apply Theorem ~\ref{thm:twist} and Corollary ~\ref{cor:twist} to cobraided bialgebra structures on group bialgebras \cite[Example 2.2.5]{majid}.  Let $G$ be a group, and let $\kg$ be its group bialgebra \cite[p.58, Example 2.4]{abe}.  Its comultiplication is determined by $\Delta(u) = u \otimes u$ for $u \in G$.

Let us first describe a (not-necessarily invertible) cobraiding form on $\kg$ (Example ~\ref{ex:dqtbialgebra}).  Thinking of $\kg \otimes \kg$ as $\bk G^2$, where $G^2 = G \times G$, a bilinear form on $\kg$ is equivalent to a function $R \colon G^2 \to \bk$.  The axioms \eqref{axiom1} and \eqref{axiom2} (with $\alpha = Id$) are equivalent to
\begin{equation}
\label{kgax12}
R(uv,w) = R(u,w)R(v,w)\quad\text{and}\quad R(u,vw) = R(u,w)R(u,v)
\end{equation}
for all $u,v,w \in G$.  The axiom ~\eqref{axiom3} is equivalent to
\[
vuR(u,v) = R(u,v)uv
\]
in $\kg$ for all $u,v \in G$, which in turn is equivalent to
\begin{equation}
\label{kgax3}
R(u,v) = 0 \quad\text{if}\quad uv \not= vu.
\end{equation}
Note that if $G$ is an abelian group, then \eqref{kgax3} holds automatically.   So a (not-necessarily invertible) cobraiding form on $\kg$ is equivalent to a function $R \colon G^2 \to \bk$ that satisfies \eqref{kgax12} and \eqref{kgax3}.  Fix such a cobraiding form $R$ for the rest of this example, and consider the cobraided bialgebra $(\kg,R)$.

Let $\alpha \colon G \to G$ be a group morphism.  Then it extends naturally to a bialgebra morphism $\alpha \colon \kg \to \kg$ with $\alpha\left(\sum_u c_u u\right) = \sum_u c_u\alpha(u)$.  By Theorem ~\ref{thm:twist} there is a cobraided Hom-bialgebra
\begin{equation}
\label{kgalpha}
\kg_\alpha = (\kg,\mu_\alpha = \alpha \circ \mu,\Delta_\alpha = \Delta \circ \alpha, \alpha,R),
\end{equation}
in which the twisted comultiplication is determined by $\Delta_\alpha(u) = \alpha(u) \otimes \alpha(u)$ for $u \in G$.

Suppose, in addition, that the group morphism $\alpha \colon G \to G$ is injective.  Then the induced bialgebra morphism $\alpha \colon \kg \to \kg$ is also injective.  By Corollary ~\ref{cor:twist} there is a cobraided Hom-bialgebra
\begin{equation}
\label{kgalphan}
\kg_\alpha^{(n)} = (\kg,\mu_\alpha,\Delta_\alpha,\alpha,\rn)
\end{equation}
for each $n \geq 1$, where $\rn = R \circ (\alpha^n \otimes \alpha^n)$.  In other words, the twisted Hom-cobraiding form $\rn$ is the function $G^2 \to \bk$ given by
\[
\rn(u,v) = R(\alpha^n(u),\alpha^n(v))
\]
for $u,v \in G$.  We will revisit this example in Example ~\ref{ex:groupbi2} below.
\qed
\end{example}

The following examples are interesting special cases of Example ~\ref{ex:groupbi}.

\begin{example}[\textbf{Anyon-generating Hom-quantum groups}]
\label{ex:anyon}
In this example, we apply Theorem ~\ref{thm:twist} and Corollary ~\ref{cor:twist} to the anyon-generating quantum group \cite{lm,majid92}, using part of the discussion in Example ~\ref{ex:groupbi}.  Here $\bk = \bC$ is the field of complex numbers and $G = \bZ/n$ is the additive cyclic group of order $n$.  We consider the group bialgebra $\bC\bZ/n$.

Let $\zeta$ denote the primitive $n$th root of unity $e^{2\pi i/n}$. Then the function $R \colon (\bZ/n)^2 \to \bC$ defined by
\[
R(a,b) = \zeta^{ab}
\]
for $a,b \in \{0, 1, \ldots , n-1\}$ satisfies the two conditions in \eqref{kgax12}.  Moreover, \eqref{kgax3} is automatically true because $\bZ/n$ is abelian.  Thus, we have a cobraided bialgebra $(\bC\bZ/n,R)$, called the \emph{anyon-generating quantum group} (\cite{lm,majid92} and \cite[p.53]{majid}).

Fix an integer $k \in \left\{1, \ldots , n-1\right\}$.  There is a bialgebra endomorphism $\alpha = \alpha_k$ on $\bC\bZ/n$, extending the group morphism
\[
\alpha_k \colon \bZ/n \to \bZ/n, \quad \alpha_k(1) = k.
\]
By Theorem ~\ref{thm:twist} there is a cobraided Hom-bialgebra
\begin{equation}
\label{cz}
(\bC\bZ/n)_\alpha = (\bC\bZ/n,\mu_\alpha = \alpha_k \circ \mu, \Delta_\alpha= \Delta \circ \alpha_k,\alpha_k,R),
\end{equation}
in which the twisted comultiplication is determined by $\Delta_\alpha(a) = ka \otimes ka$ for $a \in \bZ/n$.

Suppose, in addition, that $k$ and $n$ are relatively prime, so the map $\alpha_k$ is bijective (and hence injective).  By Corollary ~\ref{cor:twist} there is a cobraided Hom-bialgebra
\begin{equation}
\label{czt}
(\bC\bZ/n)_\alpha^{(t)} = (\bC\bZ/n,\mu_\alpha = \alpha_k \circ \mu, \Delta_\alpha= \Delta \circ \alpha_k,\alpha_k,\rt)
\end{equation}
for each $t \geq 1$, where $\rt = R \circ (\alpha^t \otimes \alpha^t)$.  More explicitly, the twisted Hom-cobraiding form $\rt$ is the function $(\bZ/n)^2 \to \bC$ given by
\begin{equation}
\label{rtzn}
\rt(a,b) = R(\alpha^t(a),\alpha^t(b)) = R(k^ta,k^tb) = \zeta^{k^{2t}ab}
\end{equation}
for $a,b \in \{0, 1, \ldots , n-1\}$.  We will revisit this example in Example ~\ref{ex:groupbi2} below.\qed
\end{example}

\begin{example}[\textbf{Integral anyon-generating Hom-quantum groups}]
\label{ex:Z}
This example is an integral version of Example ~\ref{ex:anyon}.  Here $G = \bZ$ is the infinite additive cyclic group, and $\bk$ can be any commutative ring of characteristic $0$.  Fix an invertible scalar $q$ in $\bk$.  Consider the group bialgebra $\bk\bZ$.

The function $R \colon \bZ^2 \to \bk$ defined by
\[
R(m,n) = q^{mn}
\]
satisfies \eqref{kgax12}, and \eqref{kgax3} is again trivially true because $\bZ$ is abelian.  Thus, we have a cobraided bialgebra $(\bk\bZ,R)$, denoted by $\bZ_q$ \cite[p.53 and p.537]{majid}.

Fix any non-zero integer $k$.  There is an injective bialgebra endomorphism $\alpha = \alpha_k$ on $\bk\bZ$, extending the injective group morphism
\[
\alpha_k \colon \bZ \to \bZ,\quad \alpha_k(1) = k.
\]
By Theorem ~\ref{thm:twist} and Corollary ~\ref{cor:twist}, there is a cobraided Hom-bialgebra
\[
(\bZ_q)_\alpha^{(t)} = (\bk\bZ,\mu_\alpha = \alpha_k \circ \mu,\Delta_\alpha = \Delta \circ \alpha_k,\alpha_k,\rt)
\]
for each $t \geq 0$, where $\rt = R \circ (\alpha^t \otimes \alpha^t)$ with $\alpha^0 = Id$.  As in \eqref{rtzn}, the twisted Hom-cobraiding form $\rt$ is the function $\bZ^2 \to \bk$ given by
\[
\rt(m,n) = q^{k^{2t}mn}
\]
for $m,n \in \bZ$.\qed
\end{example}

\section{Duality with braided Hom-bialgebras}
\label{sec:duality}

The purpose of this section is to establish a duality correspondence between finite dimensional braided Hom-bialgebras (Definition ~\ref{def:braided}) and finite dimensional cobraided Hom-bialgebras (Definition ~\ref{def:dqt}).  This result provides further examples of cobraided Hom-bialgebras.  As examples, we consider two finite dimensional quotients of the quantum enveloping algebra $\uq$ (Examples ~\ref{ex:uqr} and \ref{ex:smalluq}).  Throughout this section, $\bk$ is a field of characteristic $0$.

Let us first recall the definition of a braided Hom-bialgebra.

\begin{definition}
\label{def:braided}
A \textbf{braided Hom-bialgebra} is a quintuple $(A,\mu,\Delta,\alpha,R)$ in which $(A,\mu,\Delta,\alpha)$ is a Hom-bialgebra (Definition ~\ref{def:homas}) and $R \in A^{\otimes 2}$ satisfies the following three axioms:
\begin{subequations}
\label{qtaxioms}
\begin{align}
(\Delta \otimes \alpha)(R) &= R_{13}R_{23},\label{R1323}\\
(\alpha \otimes \Delta)(R) &= R_{13}R_{12},\label{R1312}\\
[(\tau\circ\Delta)(x)]R &= R\Delta(x)\label{RDelta}
\end{align}
\end{subequations}
for all $x \in A$.  Here if $R = \sum s_i \otimes t_i$, then $R_{13}R_{23}$ and $R_{13}R_{12}$ are defined as the $3$-tensors:
\begin{equation}
\label{R13}
\begin{split}
R_{13}R_{23} &= \sum \alpha(s_i) \otimes \alpha(s_j) \otimes t_it_j,\\
R_{13}R_{12} &= \sum s_is_j \otimes \alpha(t_j) \otimes \alpha(t_i).
\end{split}
\end{equation}
We call $R$ the \textbf{Hom-braiding element}.  We say that $R$ is \textbf{$\alpha$-invariant} if $\alpha^{\otimes 2}(R) = R$.
\end{definition}

Some remarks are in order.

\begin{remark}
\label{rk:braided}
\begin{enumerate}
\item
A \textbf{braided bialgebra} $(A,R)$, called a quasi-triangular bialgebra in \cite{dri87}, consists of a bialgebra $A$ and an element $R \in A^{\otimes 2}$ such that the three conditions in \eqref{qtaxioms} are satisfied for $\alpha = Id$.  The element $R$ is called the \textbf{universal $R$-matrix} or the \textbf{braiding element}.  In the literature, one usually assumes that $R$ is invertible.  However, in this paper and \cite{yau8}, we never have to use the invertibility of the braiding elements.
\item
A braided Hom-bialgebra is slightly more general than a quasi-triangular Hom-bialgebra in \cite{yau8}, which is also assumed to have a specific \emph{weak unit} $c \in A$.  The concept of a weak unit was introduced in \cite{fg2}. In a Hom-associative algebra (Definition ~\ref{def:homas}), a weak unit is an element $c \in A$ such that $\alpha(x) = cx = xc$ for all $x \in A$.  Given such a weak unit, one can define the $3$-tensors,
\[
R_{12} = R \otimes c,\quad R_{23} = c \otimes R,\quad R_{13} = (\tau \otimes Id)(R_{23}),
\]
generalizing the corresponding elements in the QYBE \eqref{qybe}.  In this case, the products $R_{13}R_{23}$ and $R_{13}R_{12}$ (computed in each tensor factor using the multiplication in $A$) coincide with the elements \eqref{R13} given in Definition ~\ref{def:braided}.

Dual to a weak unit is the concept of a weak counit.  For a Hom-coassociative coalgebra $(C,\Delta,\alpha)$, one defines a \emph{weak counit} as a linear map $\eta \colon C \to \bk$ such that
\[
(\eta \otimes Id) \circ \Delta = \alpha = (Id \otimes \eta) \circ \Delta
\]
as maps $C \to C$.  The duality correspondence (Theorem ~\ref{thm:dual}) restricts to a correspondence between finite dimensional quasi-triangular Hom-bialgebras and finite dimensional cobraided Hom-bialgebras with a weak counit.
\end{enumerate}
\end{remark}

To describe the duality correspondence of finite dimensional (co)braided Hom-bialgebras, first recall the double-dual isomorphism.  If $V$ is a finite dimensional $\bk$-vector space, then there is a linear isomorphism from $V$ to its double-dual $(V^*)^* = \Hom(V^*,\bk)$, given by evaluation:
\begin{equation}
\label{eval}
v \mapsto \langle -,v\rangle
\end{equation}
for $v \in V$.  With a slight abuse of notation, we write $\langle -,v\rangle \in (V^*)^*$ also as $v$.  Also recall from Example ~\ref{ex:duality} that, if $(A,\mu,\Delta,\alpha)$ is a finite dimensional Hom-bialgebra, then so is $(A^*,\Delta^*,\mu^*,\alpha^*)$, where $\Delta^*$, $\mu^*$, and $\alpha^*$ are defined in \eqref{deltadual} and \eqref{mudual}.

We are now ready for the promised duality correspondence.

\begin{theorem}
\label{thm:dual}
Let $(A,\mu,\Delta,\alpha)$ be a finite dimensional Hom-bialgebra and $R$ be an element in $A^{\otimes 2}$.  Then $(A,\mu,\Delta,\alpha,R)$ is a braided Hom-bialgebra if and only if $(A^*,\Delta^*,\mu^*,\alpha^*,R)$ is a cobraided Hom-bialgebra.  Moreover, $R = \alpha^{\otimes 2}(R) \in A^{\otimes 2}$ if and only if $R = R \circ (\alpha^*)^{\otimes 2} \in ((A^*)^{\otimes 2})^*$.
\end{theorem}

\begin{proof}
To prove the first assertion, we must show that $R \in A^{\otimes 2}$ is a Hom-braiding element if and only if $R \in ((A^*)^{\otimes 2})^* \cong ((A^{\otimes 2})^*)^* \cong  A^{\otimes 2}$ is a Hom-cobraiding form (Definition ~\ref{def:dqt}).  We show that the three conditions in \eqref{qtaxioms} correspond to those in \eqref{dqtaxioms}.

Let us begin with \eqref{axiom1}.  Write $R = \sum s_i \otimes t_i$, and pick arbitrary elements $\phi, \psi, \chi \in A^*$.  Then the left-hand side in \eqref{axiom1} for $A^*$ is:
\[
\begin{split}
R(\Delta^*(\phi,\psi) \otimes \alpha^*(\chi))
&= \langle \Delta^*(\phi,\psi) \otimes \alpha^*(\chi), s_i \otimes t_i \rangle\quad\text{by \eqref{eval}}\\
&= \langle \Delta^*(\phi,\psi), s_i\rangle\langle \alpha^*(\chi), t_i\rangle\\
&= \langle \phi \otimes \psi, \Delta(s_i)\rangle \langle \chi, \alpha(t_i)\rangle\quad\text{by \eqref{deltadual}}\\
&= \langle \phi \otimes \psi \otimes \chi, (\Delta \otimes \alpha)(R)\rangle.
\end{split}
\]
Similarly, the right-hand side in \eqref{axiom1} for $A^*$ is:
\[
\begin{split}
R(\alpha^*(\phi) \otimes \chi_1)R(\alpha^*(\psi) \otimes \chi_2)
&= \langle \alpha^*(\phi) \otimes \chi_1, s_i \otimes t_i\rangle \langle \alpha^*(\psi) \otimes \chi_2, s_j \otimes t_j\rangle\quad\text{by \eqref{eval}}\\
&= \langle \phi, \alpha(s_i)\rangle \langle \psi, \alpha(s_j)\rangle \langle \mu^*(\chi), t_i \otimes t_j\rangle\\
&= \langle \phi \otimes \psi \otimes \chi, \alpha(s_i) \otimes \alpha(s_j) \otimes t_it_j\rangle\quad\text{by \eqref{mudual}}\\
&= \langle \phi \otimes \psi \otimes \chi, R_{13}R_{23}\rangle.
\end{split}
\]
We conclude that \eqref{axiom1} holds for $A^*$ if and only if
\[
\langle \phi \otimes \psi \otimes \chi, (\Delta \otimes \alpha)(R)\rangle = \langle \phi \otimes \psi \otimes \chi, R_{13}R_{23}\rangle
\]
for all $\phi,\psi,\chi \in A^*$, which is equivalent to \eqref{R1323} for $A$.

An argument similar to the one in the previous paragraph shows that \eqref{axiom2} for $A^*$ is equivalent to \eqref{R1312} for $A$.

Finally, for $\phi,\psi \in A^*$ and $x \in A$, the left-hand side in \eqref{axiom3} for $A^*$ (evaluated at $x$) is:
\[
\begin{split}
\langle \Delta^*(\psi_1,\phi_1)R(\phi_2 \otimes \psi_2), x\rangle
&= \langle \psi_1 \otimes \phi_1, \Delta(x)\rangle \langle \phi_2 \otimes \psi_2, s_i \otimes t_i\rangle \quad\text{by \eqref{deltadual} and \eqref{eval}}\\
&= \langle \mu^*(\phi), x_2 \otimes s_i\rangle \langle \mu^*(\psi), x_1 \otimes t_i\rangle\\
&= \langle \phi \otimes \psi, x_2s_i \otimes x_1t_i\rangle \quad\text{by \eqref{mudual}}\\
&= \langle \phi \otimes \psi, [(\tau \circ \Delta)(x)]R\rangle.
\end{split}
\]
Similarly, the right-hand side in \eqref{axiom3} for $A^*$ (evaluated at $x$) is:
\[
\begin{split}
\langle R(\phi_1 \otimes \psi_1)\Delta^*(\phi_2,\psi_2), x\rangle
&= \langle \phi_1 \otimes \psi_1, s_i \otimes t_i\rangle \langle \phi_2 \otimes \psi_2, x_1 \otimes x_2\rangle \quad\text{by \eqref{deltadual} and \eqref{eval}}\\
&= \langle \mu^*(\phi), s_i \otimes x_1\rangle \langle \mu^*(\psi), t_i \otimes x_2\rangle\\
&= \langle \phi \otimes \psi, s_ix_1 \otimes t_ix_2\rangle\quad\text{by \eqref{mudual}}\\
&= \langle \phi \otimes \psi, R\Delta(x)\rangle.
\end{split}
\]
Thus, \eqref{axiom3} holds for $A^*$ if and only if
\[
\langle \phi \otimes \psi, [(\tau \circ \Delta)(x)]R\rangle = \langle \phi \otimes \psi, R\Delta(x)\rangle
\]
for all $\phi,\psi \in A^*$ and $x \in A$.  This is equivalent to \eqref{RDelta} for $A$.  This finishes the proof of the first assertion.

For the second assertion, we have that $R = R \circ (\alpha^*)^{\otimes 2} \in ((A^*)^{\otimes 2})^*$ if and only if
\[
R(\phi \otimes \psi) = R(\alpha^*(\phi) \otimes \alpha^*(\psi))
\]
for all $\phi, \psi \in A^*$.  This in turn is equivalent to
\[
\langle \phi \otimes \psi, R\rangle = \langle \phi \otimes \psi, \alpha^{\otimes 2}(R)\rangle
\]
for all $\phi, \psi \in A^*$, which is equivalent to $R = \alpha^{\otimes 2}(R)$.
\end{proof}

By Theorem ~\ref{thm:dual} the linear dual of a finite dimensional braided Hom-bialgebra is a cobraided Hom-bialgebra, and vice versa.  Moreover, the notion of $\alpha$-invariance is preserved by this duality. This result gives us another way to construct (co)braided Hom-bialgebras.  We now give some examples of finite dimensional (co)braided Hom-bialgebras.

\begin{example}[\textbf{Hom-quantum group bialgebras revisited}]
\label{ex:groupbi2}
This is a continuation of Examples ~\ref{ex:groupbi} and ~\ref{ex:anyon}.  If $G$ is a finite group, then the group bialgebra $\kg$ is finite dimensional.  Thus, $\kg_\alpha$ \eqref{kgalpha} and $\kg_\alpha^{(n)}$ \eqref{kgalphan} for $n \geq 1$ are all finite dimensional cobraided Hom-bialgebras.  It follows from Theorem ~\ref{thm:dual} that their duals are braided Hom-bialgebras.  This applies, in particular, to the cobraided Hom-bialgebras $(\bC\bZ/n)_\alpha$ \eqref{cz} and $(\bC\bZ/n)_\alpha^{(t)}$ \eqref{czt} in Example ~\ref{ex:anyon}.\qed
\end{example}

In each of the following two examples, we construct an infinite family of  braided Hom-bialgebras from a finite dimensional quotient of the quantum enveloping algebra $\uq$ \cite{dri87,kr,skl3}.

\begin{example}[\textbf{Reduced Hom-quantum enveloping algebras I}]
\label{ex:uqr}
Let us first recall the braided bialgebra (Remark ~\ref{rk:braided}) $\uqr$ (\cite{majid92} and \cite[Example 3.4.1]{majid}).  The ground field is $\bC$.  Let $r > 1$ be an integer and $q \in \bC$ be a primitive $(2r)$th root of unity.  As a unital $\bC$-algebra, $\uqr$ is generated by variables $X_+$, $X_-$, $K$, and $K^{-1}$, where $K$ and $K^{-1}$ are inverses of each other, with relations:
\begin{equation}
\label{uqrrelations}
\begin{split}
KX_{\pm}K^{-1} &= q^{\pm 1}X_{\pm},\quad [X_+,X_-] = \frac{K^2 - K^{-2}}{q - q^{-1}},\\
K^{4r} &= 1,\quad X_{\pm}^r = 0.
\end{split}
\end{equation}
These relations imply that the algebra $\uqr$ is finite dimensional.  It becomes a bialgebra when equipped with the comultiplication determined by
\begin{equation}
\label{uqrdelta}
\Delta(X_{\pm}) = X_{\pm} \otimes K + K^{-1} \otimes X_{\pm},\quad \Delta(K^{\pm 1}) = K^{\pm 1} \otimes K^{\pm 1}.
\end{equation}
Moreover, $\uqr$ is a braided bialgebra whose braiding element is
\begin{equation}
\label{uqrR}
R = R_K \cdot \sum_{m=0}^{r-1} \frac{(1-q^{-2})^m}{(m)_{q^{-2}}!}(KX_+)^m \otimes (K^{-1}X_-)^m,
\end{equation}
where
\begin{equation}
\label{RK}
\begin{split}
R_K &= \frac{1}{4r}\sum_{a,b=0}^{4r-1} q^{-ab/2}K^a \otimes K^b,\\
(n)_{q^{-2}} &= 1 + q^{-2} + q^{-4} + \cdots + q^{-2(n-1)},\quad
(m)_{q^{-2}}! = (m)_{q^{-2}} \cdots (2)_{q^{-2}} (1)_{q^{-2}},
\end{split}
\end{equation}
with $(0)_{q^{-2}}! \equiv 1$.

To construct Hom versions of the braided bialgebra $\uqr$, let us recall a twisting procedure from \cite[Theorem 3.1]{yau8}: If $(A,\mu,\Delta,R)$ is a braided bialgebra and $\alpha \colon A \to A$ is a bialgebra morphism, then $A_\alpha = (A,\mu_\alpha = \alpha \circ \mu,\Delta_\alpha = \Delta \circ \alpha,R)$ is a braided Hom-bialgebra.  In other words, this is the dual of Theorem ~\ref{thm:twist}.

Now we apply this twisting procedure to the braided bialgebra $\uqr$.  Let $\lambda \in \bC$ be any invertible scalar.  Consider the map $\alpha = \alpha_\lambda \colon \uqr \to \uqr$ defined by
\begin{equation}
\label{alphauqr}
\alpha(K^{\pm 1}) = K^{\pm 1},\quad \alpha(X_{\pm}) = \lambda^{\pm 1}X_{\pm}.
\end{equation}
To see that $\alpha$ defines a bialgebra morphism on $\uqr$, first note that $\alpha$ preserves all the relations in \eqref{uqrrelations}.  So $\alpha$ defines an algebra morphism on $\uqr$.  It follows from \eqref{uqrdelta} that $\alpha^{\otimes 2} \circ \Delta$ coincides with $\Delta \circ \alpha$ when applied to the generators $X_{\pm}$ and $K^{\pm 1}$.  Since both $\alpha$ and $\Delta$ are algebra morphisms, this implies that $\alpha$ is also a coalgebra morphism, and hence a bialgebra morphism.

Thus, for each invertible scalar $\lambda$, we have a braided Hom-bialgebra
\begin{equation}
\label{uqralpha}
\uqralpha = (\uqr,\mu_\alpha = \alpha \circ \mu,\Delta_\alpha = \Delta \circ \alpha,\alpha,R).
\end{equation}
The twisted comultiplication is determined by
\[
\Delta_\alpha(X_{\pm}) = \lambda^{\pm1}(X_{\pm} \otimes K + K^{-1} \otimes X_{\pm}),\quad \Delta_\alpha(K^{\pm 1}) = K^{\pm 1} \otimes K^{\pm 1}.
\]
Letting $\lambda$ run through the invertible elements in $\bC$, we can consider $\uqralpha$ as an infinite, $1$-parameter family of braided Hom-bialgebra twistings of the braided bialgebra $\uqr$.  By Theorem ~\ref{thm:dual} the dual $\uqralpha^*$ of each $\uqralpha$ is a cobraided Hom-bialgebra.

Finally, observe that the element $R$ \eqref{uqrR} is $\alpha$-invariant.  Indeed, first consider the factor $R_K$ \eqref{RK}.   Since $\alpha(K) = K$, it follows that $\alpha^{\otimes 2}(R_K) = R_K$.  For the other factor in $R$, we have
\[
\begin{split}
\alpha((KX_+)^m) \otimes \alpha((K^{-1}X_-)^m)
&= (\lambda KX_+)^m \otimes (\lambda^{-1}K^{-1}X_-)^m\\
&= (KX_+)^m \otimes (K^{-1}X_-)^m.
\end{split}
\]
This shows that $\alpha^{\otimes 2}$ fixes the other factor in $R$ as well, so $R$ is $\alpha$-invariant.  By Theorem ~\ref{thm:dual} the bilinear form $R$ on the dual $\uqralpha^*$ is also $\alpha$-invariant.
\qed
\end{example}

\begin{example}[\textbf{Reduced Hom-quantum enveloping algebras II}]
\label{ex:smalluq}
This example is about another finite dimensional quotient $\smalluq$ of $\uq$ (\cite{lm} and \cite[Example 3.4.3]{majid}).  Working over $\bC$, let $l > 1$ be an odd integer and $q \in \bC$ be a primitive $l$th root of unity.  The unital $\bC$-algebra $\smalluq$ is generated by the variables $E$, $F$, $g$, and $g^{-1}$, where $g$ and $g^{-1}$ are inverses of each other, with relations:
\begin{equation}
\label{smalluqrel}
\begin{split}
gE &= q^2Eg,\quad gF = q^{-2}Fg,\quad [E,F] = \frac{g-g^{-1}}{q-q^{-1}},\\
g^l &= 1,\quad E^l = 0 = F^l.
\end{split}
\end{equation}
These relations imply that the algebra $\smalluq$ is finite dimensional. It becomes a bialgebra when equipped with the comultiplication determined by
\begin{equation}
\label{smalluqdelta}
\begin{split}
\Delta(E) &= E \otimes g + 1 \otimes E,\quad \Delta(F) = F \otimes 1 + g^{-1} \otimes F,\\
\Delta(g^{\pm 1}) &= g^{\pm 1} \otimes g^{\pm 1}.
\end{split}
\end{equation}
Moreover, the bialgebra $\smalluq$ is a braided bialgebra whose braiding element is
\begin{equation}
\label{smalluqR}
R = \left(\frac{1}{l} \sum_{a,b=0}^{l-1} q^{-2ab} g^a \otimes g^b\right)\left(\sum_{n=0}^{l-1} \frac{(q-q^{-1})^n}{(n)_{q^{-2}}!} E^n \otimes F^n\right),
\end{equation}
where $(n)_{q^{-2}}!$ is defined in \eqref{RK}.

Now we twist the braided bialgebra $\smalluq$ into an infinite family of braided Hom-bialgebras.  For each invertible scalar $\lambda \in \bC$, consider the map $\alpha = \alpha_\lambda \colon \smalluq \to \smalluq$ defined by
\[
\alpha(g^{\pm 1}) = g^{\pm 1},\quad \alpha(E) = \lambda E,\quad \alpha(F) = \lambda^{-1}F.
\]
It follows from the relations \eqref{smalluqrel} and the definition \eqref{smalluqdelta} of $\Delta$ that $\alpha$ is a well-defined bialgebra morphism.  As in the previous example, we thus have a braided Hom-bialgebra
\begin{equation}
\label{smalluqalpha}
\smalluq_\alpha = (\smalluq,\mu_\alpha = \alpha \circ\mu,\Delta_\alpha = \Delta \circ \alpha,\alpha,R),
\end{equation}
in which the twisted comultiplication is determined by
\[
\begin{split}
\Delta_\alpha(E) &= \lambda(E \otimes g + 1 \otimes E),\quad
\Delta_\alpha(F) = \lambda^{-1}(F \otimes 1 + g^{-1} \otimes F),\\
\Delta_\alpha(g^{\pm 1}) &= g^{\pm 1} \otimes g^{\pm 1}.
\end{split}
\]
Moreover, the element $R$ \eqref{smalluqR} is $\alpha$-invariant because $\alpha(g) = g$ and $\alpha(E^n) \otimes \alpha(F^n) = E^n \otimes F^n$.

Letting $\lambda$ run through the invertible elements in $\bC$, we can consider $\smalluq_\alpha$ as an infinite, $1$-parameter family of braided Hom-bialgebra twistings of the braided bialgebra $\smalluq$.  By Theorem ~\ref{thm:dual} the dual $\smalluq_\alpha^*$ of each $\smalluq_\alpha$ is a cobraided Hom-bialgebra with $\alpha$-invariant $R$.
\qed
\end{example}

\section{Solutions of the HYBE from cobraided Hom-bialgebras}
\label{sec:hybe}

Each comodule $V$ of a cobraided bialgebra has a canonical solution $B_{V,V}$ \eqref{BVV} of the Yang-Baxter equation \eqref{ybe}.  In this section, we generalize this important fact to the Hom setting.  In particular, we show that every comodule over a cobraided Hom-bialgebra with an $\alpha$-invariant Hom-cobraiding form (Definition ~\ref{def:dqt}) has a canonical solution of the Hom-Yang-Baxter equation (HYBE) (Corollary ~\ref{cor:comod}).  Let us first recall the HYBE.

\begin{definition}[\cite{yau5}]
\label{def:hybe}
Let $V$ be a $\bk$-module and $\alpha \colon V \to V$ be a linear map.  The \textbf{Hom-Yang-Baxter equation} (HYBE) is defined as
\begin{equation}
\label{hybe}
(\alpha \otimes B) \circ (B \otimes \alpha) \circ (\alpha \otimes B) = (B \otimes \alpha) \circ (\alpha \otimes B) \circ (B \otimes \alpha),
\end{equation}
where $B \colon V^{\otimes 2} \to V^{\otimes 2}$ is a bilinear map that commutes with $\alpha^{\otimes 2}$.  In this case, we say that $B$ is a solution of the HYBE for $(V,\alpha)$.
\end{definition}

The YBE \eqref{ybe} is the special case of the HYBE \eqref{hybe} in which $\alpha = Id$.  Solutions of the YBE are important because, among other properties, they give rise to braid group representations.  This is still true for solutions of the HYBE.  In fact, it is shown in \cite[Theorem 1.4]{yau5} that every solution of the HYBE gives rise to operators that satisfy the braid relations \cite{artin2,artin}.  When $\alpha$ and $B$ are both invertible, these operators give rise to a corresponding representation of the braid group.  Many examples of solutions of the HYBE can be found in \cite{yau5,yau6}. Moreover, it is shown in \cite{yau8} that each module over a braided Hom-bialgebra with an $\alpha$-invariant braiding element has a non-trivial solution of the HYBE.

In order to generalize the solutions of the YBE associated to comodules of cobraided bialgebras, we need a suitable notion of comodules in the Hom setting.

\begin{definition}
\label{def:comodule}
\begin{enumerate}
\item
A \textbf{Hom-module} is a pair $(V,\alpha)$ consisting of a $\bk$-module $V$ and a linear map $\alpha$.
\item
Let $(C,\Delta,\alpha_C)$ be a Hom-coassociative coalgebra (Definition ~\ref{def:homas}).  By a \textbf{$C$-comodule} we mean a Hom-module $(M,\alpha_M)$ together with a linear map $\rho \colon M \to C \otimes M$ (the structure map) such that
\begin{equation}
\label{homcoass}
(\Delta \otimes \alpha_M) \circ \rho = (\alpha_C \otimes \rho) \circ \rho
\end{equation}
and
\begin{equation}
\label{comult}
(\alpha_C \otimes \alpha_M) \circ \rho = \rho \circ \alpha_M.
\end{equation}
The conditions \eqref{homcoass} and \eqref{comult} are referred to as \emph{Hom-coassociativity} and \emph{comultiplicativity}, respectively.
\end{enumerate}
\end{definition}


\begin{example}
\begin{enumerate}
\label{ex:comod}
\item
If $(C,\Delta,\alpha)$ is a Hom-coassociative coalgebra, then $(C,\alpha)$ is a $C$-comodule with structure map $\rho = \Delta$.
\item
Let $(C,\Delta)$ be a coassociative coalgebra and $M$ be a $C$-comodule (in the usual sense) with structure map $\rho \colon M \to C \otimes M$.  Suppose that $\alpha_C \colon C \to C$ is a coalgebra morphism and $\alpha_M \colon M \to M$ is a linear map such that $(\alpha_C \otimes \alpha_M) \circ \rho = \rho \circ \alpha_M$. Then $(M,\alpha_M)$ is a $C_\alpha$-comodule with structure map $\rho_\alpha = \rho \circ \alpha_M$, where $C_\alpha$ is the Hom-coassociative coalgebra $(C,\Delta_\alpha = \Delta \circ \alpha_C,\alpha_C)$ (Example ~\ref{ex:homas}).\qed
\end{enumerate}
\end{example}

To state the main results of this section, we need some notations.  Let $(A,\mu,\Delta,\alpha,R)$ be a cobraided Hom-bialgebra (Definition ~\ref{def:dqt}), and let $V$ be an $A$-comodule with structure map $\rho_V \colon V \to A \otimes V$.  We use the notation $\rho_V(v) = \sum v_A \otimes v_V$ for $v \in V$.  Let $W$ be another $A$-comodule with structure map $\rho_W$.  Essentially as in \eqref{BVV}, define the map $B_{V,W} \colon V \otimes W \to W \otimes V$ by
\begin{equation}
\label{BVW}
B_{V,W}(v \otimes w) = \sum R(w_A \otimes v_A)w_W \otimes v_V
\end{equation}
for $v \in V$ and $w \in W$.  In other words, the map $B_{V,W}$ is defined by insisting that the following diagram be commutative:
\[
\SelectTips{cm}{10}
\xymatrix{
V \otimes W \ar[rr]^-{\tau} \ar[d]_-{B_{V,W}} & & W \otimes V \ar[rr]^-{\rho_W \otimes \rho_V} & & A \otimes W \otimes A \otimes V \ar[d]^-{Id \otimes \tau \otimes Id}\\
W \otimes V & & \bk \otimes W \otimes V \ar[ll]_-{\cong} & & A \otimes A \otimes W \otimes V \ar[ll]_-{R \otimes Id \otimes Id}.
}
\]
Recall that the Hom-cobraiding form $R$ is said to be $\alpha$-invariant if $R = R \circ \alpha^{\otimes 2}$.

\begin{theorem}
\label{thm:comod}
Let $(A,\mu,\Delta,\alpha,R)$ be a cobraided Hom-bialgebra with an $\alpha$-invariant Hom-cobraiding form $R$.  Let $U$, $V$, and $W$ be $A$-comodules.  Then:
\begin{enumerate}
\item
The operator $B_{V,W}$ \eqref{BVW} commutes with $\alpha$, i.e.,
\begin{equation}
\label{Balpha}
(\alpha_W \otimes \alpha_V) \circ B_{V,W}
= B_{V,W} \circ (\alpha_V \otimes \alpha_W).
\end{equation}
\item
The equality
\begin{equation}
\label{comodhybe}
\begin{split}
(\alpha_W \otimes B_{U,V}) & \circ (B_{U,W} \otimes \alpha_V) \circ (\alpha_U \otimes B_{V,W})\\
&= (B_{V,W} \otimes \alpha_U) \circ (\alpha_V \otimes B_{U,W}) \circ (B_{U,V} \otimes \alpha_W)
\end{split}
\end{equation}
holds.
\end{enumerate}
\end{theorem}

Before we give the proof of Theorem ~\ref{thm:comod}, let us first record the following special cases.

\begin{corollary}
\label{cor:comod}
Let $(A,\mu,\Delta,\alpha,R)$ be a cobraided Hom-bialgebra with an $\alpha$-invariant Hom-cobraiding form $R$, and let $(V,\alpha_V)$ be an $A$-comodule.  Then the map $B_{V,V} \colon V^{\otimes 2} \to V^{\otimes 2}$ \eqref{BVW} is a solution of the HYBE for $(V,\alpha_V)$.
\end{corollary}

\begin{proof}
Setting $U = V = W$ in Theorem ~\ref{thm:comod}, the equalities \eqref{Balpha} and \eqref{comodhybe} say that $B_{V,V}$ commutes with $\alpha_V^{\otimes 2}$ and that $B_{V,V}$ satisfies the HYBE \eqref{hybe}, respectively.
\end{proof}

In other words, every comodule of a cobraided Hom-bialgebra with $\alpha$-invariant $R$ has a canonical solution $B_{V,V}$ of the HYBE. We have given a number of examples of cobraided Hom-bialgebras with an $\alpha$-invariant Hom-cobraiding form, for which Corollary ~\ref{cor:comod} can be applied.   These cobraided Hom-bialgebras include the Hom-quantum matrices $\mqalpha$ \eqref{mqalpha}, the Hom-quantum general linear groups $\glalpha$ \eqref{glalpha}, the Hom-quantum special linear groups $\slqalpha$ \eqref{slalpha}, the $2$-parameter Hom-quantum groups $\mpq_\alpha$ \eqref{mpqalpha}, the non-standard Hom-quantum groups $\mqnsalpha$ \eqref{mq'alpha}, and the duals of the reduced Hom-quantum groups $\uqr_\alpha$ \eqref{uqralpha} and $\smalluq_\alpha$ \eqref{smalluqalpha}.

The following result is a useful special case of Corollary ~\ref{cor:comod}.  We will use it in Examples ~\ref{ex:shqsym} - ~\ref{ex:mhqsym}.

\begin{corollary}
\label{cor2:comod}
Let $(A,R)$ be a cobraided bialgebra and $\alpha \colon A \to A$ be a bialgebra morphism such that $R = R \circ \alpha^{\otimes 2}$.  Let $V$ be an $A$-comodule with structure map $\rho$ and $\alpha_V \colon V \to V$ be a linear map such that $(\alpha \otimes \alpha_V) \circ \rho = \rho \circ \alpha_V$.  Then the map $B_\alpha \colon V^{\otimes 2} \to V^{\otimes 2}$ defined by
\begin{equation}
\label{Balphacomod}
B_\alpha(v \otimes w) = \sum R(w_A \otimes v_A)\alpha_V(w_V) \otimes \alpha_V(v_V)
\end{equation}
for $v,w \in V$ is a solution of the HYBE for $(V,\alpha_V)$.
\end{corollary}

\begin{proof}
We know that $(V,\alpha_V)$ is an $A_\alpha$-comodule with structure map $\rho_\alpha = \rho \circ \alpha_V$ (Example ~\ref{ex:comod}), where $A_\alpha$ is the cobraided Hom-bialgebra $(A,\mu_\alpha,\Delta_\alpha,\alpha,R)$ (Theorem ~\ref{thm:twist}).  Since $R$ is $\alpha$-invariant, by Corollary ~\ref{cor:comod} there is a solution $B_{V,V}$ of the HYBE for $(V,\alpha_V)$.  The operator $B_{V,V}$ takes the form $B_\alpha$ \eqref{Balphacomod} by the $\alpha$-invariance of $R$ and because $\rho_\alpha(v) = \sum \alpha(v_A) \otimes \alpha_V(v_V)$ for $v \in V$.
\end{proof}

In Corollary ~\ref{cor2:comod}, if $\alpha = Id \colon A \to A$, then $R = R \circ \alpha^{\otimes 2}$ holds automatically.  Moreover, the condition $(Id \otimes \alpha_V) \circ \rho = \rho \circ \alpha_V$ means exactly that $\alpha_V$ is a morphism of $A$-comodules.  Therefore, we have the following special case of Corollary ~\ref{cor2:comod}.

\begin{corollary}
\label{cor3:comod}
Let $(A,R)$ be a cobraided bialgebra, $V$ be an $A$-comodule, and $\alpha_V \colon V \to V$ be a morphism of $A$-comodules.  Then the map $B_\alpha$ \eqref{Balphacomod} is a solution of the HYBE for $(V,\alpha_V)$.
\end{corollary}

The rest of this section is devoted to the proof of Theorem ~\ref{thm:comod}.

\begin{proof}[Proof of Theorem ~\ref{thm:comod}]
In order to simplify the typography, we will often omit the subscripts in the maps $\alpha_V$, etc.  To prove \eqref{Balpha}, pick $v \in V$, $w \in W$ and compute as follows:
\[
\begin{split}
(\alpha_W \otimes \alpha_V)(B_{V,W}(v \otimes w))
&= R(\alpha(w_A) \otimes \alpha(v_A)) \alpha(w_W) \otimes \alpha(v_V)\quad\text{by $\alpha$-invariance of $R$}\\
&= R(\alpha(w)_A \otimes \alpha(v)_A) \alpha(w)_W \otimes \alpha(v)_V \quad\text{by comultiplicativity \eqref{comult}}\\
&= B_{V,W}((\alpha_V \otimes \alpha_W)(v \otimes w)).
\end{split}
\]
This proves that $B_{V,W}$ commutes with $\alpha$.

To prove \eqref{comodhybe}, let $z = u \otimes v \otimes w$ be a typical generator in $U \otimes V \otimes W$.  In the two lemmas below, we will show that both sides in \eqref{comodhybe}, when applied to $z$, give
\begin{equation}
\label{theta}
\theta = R\left((v_A)_1 \otimes (u_A)_1\right)R\left((w_A)_1 \otimes (u_A)_2\right)R\left((w_A)_2 \otimes (v_A)_2\right)z'
\end{equation}
in $W \otimes V \otimes U$, where
\begin{equation}
\label{z'}
z' = \alpha^2(w_W) \otimes \alpha^2(v_V) \otimes \alpha^2(u_U).
\end{equation}
In the $3$-tensor $\theta$, the subscripts $1$ and $2$ refer to the comultiplication in $A$, i.e., $\Delta(u_A) = \sum (u_A)_1 \otimes (u_A)_2$, and so on.  It remains to prove the two lemmas below.
\end{proof}

\begin{lemma}
\label{lem1:comod}
We have
\[
(\alpha_W \otimes B_{U,V}) \circ (B_{U,W} \otimes \alpha_V) \circ (\alpha_U \otimes B_{V,W})(u \otimes v \otimes w) = \theta,
\]
where $\theta$ is defined in \eqref{theta}.
\end{lemma}

\begin{proof}
This proof is a long computation, but it is conceptually elementary.  With $z = u \otimes v \otimes w$, we first claim that
\begin{equation}
\label{lhs1}
\begin{split}
(B_{U,W} \otimes \alpha_V) & \circ (\alpha_U \otimes B_{V,W})(z)\\
&= R\left(\alpha(w_A) \otimes u_Av_A\right) \alpha(w_W) \otimes \alpha(u_U) \otimes \alpha(v_V).
\end{split}
\end{equation}
Indeed, we have:
\[
\begin{split}
&(B_{U,W} \otimes \alpha_V) \circ (\alpha_U \otimes B_{V,W})(z)\\
&= R(w_A \otimes v_A)(B_{U,W} \otimes \alpha_V)\left(\alpha(u) \otimes w_W \otimes v_V\right)\\
&= R(\alpha(w_A) \otimes \alpha(v_A))R\left((w_W)_A \otimes \alpha(u)_A\right)(w_W)_W \otimes \alpha(u)_U \otimes \alpha(v_V) \quad\text{by $\alpha$-invariance of $R$}\\
&= R\left((w_A)_1 \otimes \alpha(v_A)\right)R\left((w_A)_2 \otimes \alpha(u_A)\right) \alpha(w_W) \otimes \alpha(u_U) \otimes \alpha(v_V) \quad\text{by \eqref{homcoass} and \eqref{comult}}\\
&= R\left(\alpha(w_A) \otimes u_Av_A\right)\alpha(w_W) \otimes \alpha(u_U) \otimes \alpha(v_V) \quad\text{by \eqref{axiom2}}.
\end{split}
\]
This proves \eqref{lhs1}.

Now we apply $\alpha_W \otimes B_{U,V}$ to \eqref{lhs1}:
\[
\begin{split}
& (\alpha_W \otimes B_{U,V}) \circ (B_{U,W} \otimes \alpha_V) \circ (\alpha_U \otimes B_{V,W})(z)\\
&= R\left(\alpha(w_A) \otimes u_Av_A\right)R\left(\alpha(v_V)_A \otimes \alpha(u_U)_A\right) \alpha^2(w_W) \otimes \alpha(v_V)_V \otimes \alpha(u_U)_U\\
&= R\left(\alpha(w_A) \otimes u_Av_A\right) R\left(\alpha((v_V)_A)\otimes \alpha((u_U)_A)\right) \alpha^2(w_W) \otimes \alpha((v_V)_V) \otimes \alpha((u_U)_U) \quad\text{by \eqref{comult}}\\
&= R\left(\alpha^2(w_A) \otimes \alpha(u_A)\alpha(v_A)\right)R\left((v_V)_A \otimes (u_U)_A\right)\alpha^2(w_W) \otimes \alpha((v_V)_V) \otimes \alpha((u_U)_U)\\
&= R\left(\alpha^2(w_A) \otimes (u_A)_1 (v_A)_1\right)R\left((v_A)_2 \otimes (u_A)_2\right)\alpha^2(w_W) \otimes \alpha^2(v_V) \otimes \alpha^2(u_U) \quad\text{by \eqref{homcoass}}\\
&= R\left(\alpha(w_A)_1 \otimes \alpha((v_A)_1)\right)R\left(\alpha(w_A)_2 \otimes \alpha((u_A)_1)\right)R\left((v_A)_2 \otimes (u_A)_2\right)z'\quad\text{by \eqref{axiom2}}.
\end{split}
\]
In the third equality above, we used the $\alpha$-invariance of $R$ and \eqref{homassaxioms}.  The $3$-tensor $z'$ is defined in \eqref{z'}.  Applying the OQHYBE \eqref{oqhybe} to the previous line, we continue the above computation:
\[
\begin{split}
&= R\left((v_A)_1 \otimes (u_A)_1\right)R\left(\alpha(w_A)_1 \otimes \alpha((u_A)_2)\right)R\left(\alpha(w_A)_2 \otimes \alpha((v_A)_2)\right)z'\\
&= R\left((v_A)_1 \otimes (u_A)_1\right)R\left(\alpha((w_A)_1) \otimes \alpha((u_A)_2)\right)R\left(\alpha((w_A)_2)\otimes \alpha((v_A)_2)\right)z'\quad\text{by \eqref{homcoassaxioms}}\\
&= R\left((v_A)_1 \otimes (u_A)_1\right)R\left((w_A)_1 \otimes (u_A)_2\right)R\left((w_A)_2 \otimes (v_A)_2\right)z'\quad\text{by $\alpha$-invariance}.
\end{split}
\]
The previous line is equal to $\theta$ \eqref{theta}, thereby finishing the proof.
\end{proof}

\begin{lemma}
\label{lem2:comod}
We have
\[
(B_{V,W} \otimes \alpha_U) \circ (\alpha_V \otimes B_{U,W}) \circ (B_{U,V} \otimes \alpha_W)(u \otimes v \otimes w) = \theta,
\]
where $\theta$ is defined in \eqref{theta}.
\end{lemma}

\begin{proof}
With $z = u \otimes v \otimes w$, first we claim that
\begin{equation}
\label{rhs1}
\begin{split}
(\alpha_V \otimes B_{U,W}) & \circ (B_{U,V} \otimes \alpha_W)(z)\\
&= R\left(v_Aw_A \otimes \alpha(u_A)\right) \alpha(v_V) \otimes \alpha(w_W) \otimes \alpha(u_U).
\end{split}
\end{equation}
Indeed, we have:
\[
\begin{split}
& (\alpha_V \otimes B_{U,W}) \circ (B_{U,V} \otimes \alpha_W)(z)\\
&= R(v_A \otimes u_A)(\alpha_V \otimes B_{U,W})(v_V \otimes u_U \otimes \alpha(w))\\
&= R(\alpha(v_A) \otimes \alpha(u_A))R\left(\alpha(w)_A \otimes (u_U)_A\right)\alpha(v_V) \otimes \alpha(w)_W \otimes (u_U)_U\quad\text{by $\alpha$-invariance}\\
&= R\left(\alpha(v_A) \otimes (u_A)_1\right)R\left(\alpha(w_A) \otimes (u_A)_2\right)\alpha(v_V) \otimes \alpha(w_W) \otimes \alpha(u_U)\quad\text{by \eqref{homcoass} and \eqref{comult}}.
\end{split}
\]
The previous line is equal to the right-hand side in \eqref{rhs1} by \eqref{axiom1}.

Now apply $B_{V,W} \otimes \alpha_U$ to \eqref{rhs1}:
\[
\begin{split}
& (B_{V,W} \otimes \alpha_U) \circ (\alpha_V \otimes B_{U,W}) \circ (B_{U,V} \otimes \alpha_W)(z)\\
&= R\left(v_Aw_A \otimes \alpha(u_A)\right)R\left(\alpha(w_W)_A \otimes \alpha(v_V)_A\right)\alpha(w_W)_W \otimes \alpha(v_V)_V \otimes \alpha^2(u_U)\\
&= R\left(v_Aw_A \otimes \alpha(u_A)\right)R\left(\alpha((w_W)_A) \otimes \alpha((v_V)_A)\right)\alpha((w_W)_W)\otimes \alpha((v_V)_V)\otimes \alpha^2(u_U)\quad\text{by \eqref{comult}}\\
&= R\left(\alpha(v_A)\alpha(w_A) \otimes \alpha^2(u_A)\right)R\left((w_W)_A \otimes (v_V)_A\right)\alpha((w_W)_W)\otimes \alpha((v_V)_V)\otimes \alpha^2(u_U)\\
&= R\left((v_A)_1(w_A)_1 \otimes \alpha^2(u_A)\right)R\left((w_A)_2 \otimes (v_A)_2\right)z' \quad\text{by \eqref{homcoass} and \eqref{z'}}\\
&= R\left(\alpha((v_A)_1) \otimes \alpha(u_A)_1\right)R\left(\alpha((w_A)_1)\otimes \alpha(u_A)_2\right)R\left((w_A)_2 \otimes (v_A)_2\right)z'\quad\text{by \eqref{axiom1}}\\
&= R\left(\alpha((v_A)_1) \otimes \alpha((u_A)_1)\right)R\left(\alpha((w_A)_1)\otimes\alpha((u_A)_2)\right)R\left((w_A)_2 \otimes (v_A)_2\right)z' \quad\text{by \eqref{homcoassaxioms}}\\
&= \theta \quad\text{by $\alpha$-invariance}.
\end{split}
\]
In the third equality above, we used the $\alpha$-invariance of $R$ and \eqref{homassaxioms}.
\end{proof}

This completes the proof of Theorem ~\ref{thm:comod}.

\section{Hom-quantum geometry}
\label{sec:coaction}

The purpose of this section is to discuss Hom-quantum geometry, focusing on Hom-type, non-associative and non-coassociative analogues of the quantum matrices coactions on the quantum planes (Example ~\ref{ex:qsym}).  We use the phrase \emph{Hom-quantum geometry} loosely to refer to Hom-type analogues of quantum group (co)actions on quantum spaces, such as the quantum planes.  The quantum matrices considered in this section are the quantum groups $\mq$ (Example ~\ref{ex:mq}) and $\mqns$ (Example ~\ref{ex:mq'}) encountered in section ~\ref{sec:twist}.  The quantum planes are quantum, non-commutative versions of the affine plane $\bk[x,y]$ (Definition ~\ref{def:qplanes}).

We give a general twisting procedure, similar to Theorem ~\ref{thm:twist}, by which comodule algebras are twisted into their Hom-type analogues (Theorem ~\ref{thm:comodalg}).  Then we use it to
construct multi-parameter, infinite families of non-associative and non-coassociative twistings of the usual $\mq$-comodule algebra structures on the standard and the fermionic quantum planes (Examples ~\ref{ex:shqsym} and ~\ref{ex:fhqsym}).  We also construct an infinite family of Hom-type twistings of the $\mqns$-comodule algebra structure on the mixed quantum plane (Example ~\ref{ex:mhqsym}).


Let us first recall the various quantum planes and the quantum matrices coactions on them.  The reader can consult \cite[IV]{kassel} for the standard quantum plane, \cite[Ch.3]{street} for the standard and the fermionic quantum planes, and \cite[4.5 and 10.2]{majid} for all three versions of the quantum planes.  For the rest of this section, $\bk$ is a field of characteristic $0$ and $q \in \bk$ is an invertible scalar with $q^2 \not= -1$.

\begin{definition}[\textbf{Quantum planes}]
\label{def:qplanes}
Let $x$ and $y$ be non-commuting variables, and let $\bk\{x,y\}$ be the free unital associative $\bk$-algebra generated by $x$ and $y$.
\begin{enumerate}
\item
The \textbf{standard quantum plane} is the $\bk$-algebra
\begin{equation}
\label{qp}
\qp = \frac{\bk\{x,y\}}{(yx - qxy)}.
\end{equation}
It has a $\bk$-linear basis $\{x^iy^j \colon i,j \geq 0\}$.
\item
The \textbf{fermionic quantum plane} is the $\bk$-algebra
\begin{equation}
\label{qf}
\qf = \frac{\bk\{x,y\}}{(x^2,\ y^2,\ xy + qyx)}.
\end{equation}
It has a $\bk$-linear basis $\{1,x,y,xy\}$.
\item
The \textbf{mixed quantum plane} is the $\bk$-algebra
\begin{equation}
\label{qm}
\qm = \frac{\bk\{x,y\}}{(y^2,\ yx - qxy)}.
\end{equation}
It has a $\bk$-linear basis $\{x^iy^\epsilon \colon i \geq 0,\ \epsilon \leq 1\}$.
\end{enumerate}
\end{definition}

In the standard and the mixed quantum planes, the two generators satisfy the quantum commutation relation, $yx = qxy$ \eqref{qcom}.  In particular, if $q = 1$, then the standard quantum plane $\qp$ is exactly the affine plane $\bk[x,y]$.  The fermionic quantum plane $\qf$ is also known as the \emph{quantum super-plane}, as in \cite[Ch.3]{street}.  Note that the mixed quantum plane is a sort of halfway transition between the standard and the fermionic quantum planes.  Indeed, $\qm$ is the quotient of the standard quantum plane $\qp$ by the relation $y^2 = 0$.  Furthermore, the fermionic quantum plane $\qf$ is the quotient of the mixed quantum plane $\qmtwisted$ by the relation $x^2 = 0$.


To describe the quantum matrices coactions on the quantum planes, let us recall the definition of a comodule algebra.  The concept of a comodule algebra is discussed in many books on Hopf algebras and quantum groups, such as \cite[p.138]{abe}, \cite[III.7]{kassel}, and \cite[Ch.4]{mont}.

\begin{definition}
\label{def:comodalg}
Let $H$ be a bialgebra and $A$ be an associative algebra.  An \textbf{$H$-comodule algebra} structure on $A$ consists of an $H$-comodule structure $\rho \colon A \to H \otimes A$ such that $\rho$ is also a morphism of algebras.
\end{definition}

If $\rho \colon A \to H \otimes A$ is an $H$-comodule structure map, we write $\rho(a) = \sum a_H \otimes a_A$ for $a \in A$.  Then $\rho$ is a morphism of algebras if and only if
\begin{equation}
\label{comodalgax}
\rho(ab) = \sum a_Hb_H \otimes a_Ab_A
\end{equation}
for all $a,b \in A$.  Also, $\rho$ is a morphism of algebras if and only if the multiplication $\mu_A$ in $A$ is an $H$-comodule morphism \cite[III.7.1 and III.7.2]{kassel}.

To construct an $H$-comodule algebra structure on an algebra $A$, one can first construct an algebra morphism $\rho \colon A \to H \otimes A$.  Then one shows that $\rho$ is an $H$-comodule structure map, i.e., $(\Delta \otimes Id)\circ \rho = (Id \otimes \rho)\circ\rho$, where $\Delta$ is the comultiplication in $H$.

Now we describe the usual quantum matrices comodule algebra structures on the quantum planes.


\begin{example}[\textbf{Quantum symmetries on the quantum planes}]
\label{ex:qsym}
\begin{enumerate}
\item
Recall the quantum group $\mq$ of quantum matrices from Example ~\ref{ex:mq}.  For $\mq$ we also assume that $q^{\frac{1}{2}}$ exists and is invertible.  There is a $\mq$-comodule algebra structure on the standard quantum plane whose structure map $\rho \colon \qp \to \mq \otimes \qp$ is determined by
\begin{equation}
\label{mqcoaction}
\rho\xymat = \abcd \otimes \xymat.
\end{equation}
In other words, we have
\[
\rho(x) = a \otimes x + b \otimes y \quad
\text{and}\quad
\rho(y) = c \otimes x + d \otimes y.
\]
One checks that $\rho$ defines a morphism of algebras by observing that it preserves the quantum commutation relation,
\[
\rho(y)\rho(x) = q\rho(x)\rho(y),
\]
using \eqref{qcom} and \eqref{six}.  That $\rho$ gives the standard quantum plane $\qp$ the structure of a $\mq$-comodule follows from the definition of the comultiplication $\Delta$ on $\mq$ \eqref{deltaMq}.  Indeed, using the matrix notations above, we have
\begin{equation}
\label{mqcomodule}
\begin{split}
(\Delta \otimes Id)\circ \rho\xymat
&= \abcd \otimes \abcd \otimes \xymat\\
&= (Id \otimes \rho) \circ \rho\xymat.
\end{split}
\end{equation}
This is enough to conclude that $(\Delta \otimes Id)\circ \rho$ and $(Id \otimes \rho) \circ \rho$ are equal because both of them are algebra morphisms.
\item
Similarly, there is a $\mq$-comodule algebra structure on the fermionic quantum plane whose structure map $\rho \colon \qf \to \mq \otimes \qf$ is given by \eqref{mqcoaction} again.  As before, $\rho$ is a morphism of algebras because it preserves the necessary relations, i.e.,
\[
\rho(x)^2 = 0 = \rho(y)^2\quad\text{and}\quad
\rho(x)\rho(y) + q\rho(y)\rho(x) = 0.
\]
That $\rho$ gives the fermionic quantum plane $\qf$ the structure of a $\mq$-comodule follows from \eqref{mqcomodule}.
\item
Recall the non-standard quantum group $\mqns$ from Example ~\ref{ex:mq'}.  There is a $\mqns$-comodule algebra structure on the mixed quantum plane whose structure map $\rho \colon \qm \to \mqns \otimes \qm$ is once again given as in \eqref{mqcoaction}.  The relations
\[
\rho(y)^2 = 0\quad\text{and}\quad\rho(y)\rho(x) = q\rho(x)\rho(y)
\]
are easy to check, so $\rho$ is a morphism of algebras.  One uses \eqref{mqcomodule} to infer that $\rho$ gives the mixed quantum plane $\qm$ the structure of a $\mqns$-comodule.
\qed
\end{enumerate}
\end{example}

In order to study Hom-type, non-(co)associative twistings of these quantum matrices comodule algebra structures on the quantum planes, let us define the Hom version of a comodule algebra.


\begin{definition}
\label{def:cha}
Let $(H,\mu_H,\Delta_H,\alpha_H)$ be a Hom-bialgebra and $(A,\mu_A,\alpha_A)$ be a Hom-associative algebra (Definition ~\ref{def:homas}).  An \textbf{$H$-comodule Hom-algebra} structure on $A$ consists of an $H$-comodule structure $\rho \colon A \to H \otimes A$ on $A$ (Definition ~\ref{def:comodule}) such that \eqref{comodalgax} is satisfied.
\end{definition}

Recall from Definition ~\ref{def:comodule} that an $H$-comodule structure map $\rho$ commutes with $\alpha$.  The condition ~\eqref{comodalgax} still makes sense in the Hom setting, and it says that $\rho$ commutes with the multiplications $\mu$ in the obvious sense.  Therefore, we can equivalently define an $H$-comodule Hom-algebra structure as an $H$-comodule structure $\rho$ that is also a morphism of Hom-associative algebras, i.e., a linear map that commutes with both $\alpha$ and $\mu$.


As in section ~\ref{sec:twist}, we now give a general twisting procedure by which a large family of comodule Hom-algebras can be constructed from any comodule algebra.  This result will be applied to the quantum symmetries on the quantum planes (Example ~\ref{ex:qsym}).

\begin{theorem}
\label{thm:comodalg}
Let $(H,\mu_H,\Delta_H)$ be a bialgebra, and let $(A,\mu_A)$ be an $H$-comodule algebra with structure map $\rho \colon A \to H \otimes A$.  Let $\alpha_H \colon H \to H$ be a bialgebra morphism, and let $\alpha_A \colon A \to A$ be an algebra morphism such that \begin{equation}
\label{eq:rhoalpha}
\rho \circ \alpha_A = (\alpha_H \otimes \alpha_A) \circ \rho.
\end{equation}
Then the Hom-associative algebra $A_\alpha = (A,\mualphaA = \alpha_A \circ \mu_A,\alpha_A)$ is an $H_\alpha$-comodule Hom-algebra with structure map $\rho_\alpha = \rho \circ \alpha_A$, where $H_\alpha$ is the Hom-bialgebra $(H,\mualphaH = \alpha_H \circ \mu_H,\Delta_\alpha = \Delta_H \circ \alpha_H,\alpha_H)$ in Example ~\ref{ex:homas}.
\end{theorem}

\begin{proof}
First we show that $\rho_\alpha$ gives $(A,\alpha_A)$ the structure of an $H_\alpha$-comodule (Definition ~\ref{def:comodule}).  That $\rho_\alpha$ is comultiplicative \eqref{comult} follows from \eqref{eq:rhoalpha}.  Next, the Hom-coassociativity \eqref{homcoass} of $\rho_\alpha$ means that
\[
(\alpha_H \otimes \rho_\alpha) \circ \rho_\alpha = (\Delta_\alpha \otimes \alpha_A) \circ \rho_\alpha.
\]
This equality is true by the following commutative diagram:
\[
\SelectTips{cm}{10}
\xymatrix{
A \ar[rr]_-{\alpha_A} \ar[dd]_-{\rho_\alpha}
\ar@{<} `u[rrrr]  `_d[rrrr] ^-{\rho_\alpha} [rrrr]
& & A \ar[rr]_-{\rho} \ar[ddll]_-{\rho} \ar[d]^-{\alpha_A} & & H \otimes A \ar[d]_-{\alpha_H \otimes \alpha_A}
\ar@/^2pc/[dd]^-{\alpha_H \otimes \rho_\alpha}\\
 & & A \ar[rr]^-{\rho} \ar[d]^-{\rho} & & H \otimes A \ar[d]_-{Id_H \otimes \rho}\\
H \otimes A \ar[rr]^-{\alpha_H \otimes \alpha_A} \ar@{<} `d[rrrr]  `[rrrr] _-{\Delta_\alpha \otimes \alpha_A} [rrrr]
& & H \otimes A \ar[rr]^-{\Delta_H \otimes Id_A} & & H^{\otimes 2} \otimes A.
}
\]
The lower-left triangle and the upper-right square are commutative by \eqref{eq:rhoalpha}.  The lower-right square is commutative because $A$ is an $H$-comodule (in the usual sense) by assumption.  We have shown that $\rho_\alpha$ gives $(A,\alpha_A)$ the structure of an $H_\alpha$-comodule.

The condition \eqref{comodalgax} (with $\rho_\alpha = \rho \circ \alpha_A$, $\mualphaA = \alpha_A \circ \mu_A$, and $\mualphaH = \alpha_H \circ \mu_H$) means that
\[
\rho_\alpha \circ \mualphaA = (\mualphaH \otimes \mualphaA) \circ (Id \otimes \tau \otimes Id) \circ \rho_\alpha^{\otimes 2}.
\]
This equality is true by the following commutative diagram, where $\nu = (\mu_H \otimes \mu_A) \circ (Id \otimes \tau \otimes Id)$ and $\eta = (\mualphaH \otimes \mualphaA) \circ (Id \otimes \tau \otimes Id)$:
\[
\SelectTips{cm}{10}
\xymatrix{
A \otimes A \ar[rr]_-{\alpha_A^{\otimes 2}} \ar[d]^-{\mu_A} \ar@/_2pc/[dd]_-{\mualphaA}
\ar `u[rrrr] `_d[rrrr] ^-{\rho_\alpha^{\otimes 2}} [rrrr]
& & A \otimes A \ar[rr]_-{\rho^{\otimes 2}} \ar[d]^-{\mu_A} & & H \otimes A \otimes H \otimes A \ar[d]_-{\nu} \ar@/^2pc/[dd]^-{\eta}\\
A \ar[rr]^-{\alpha_A} \ar[d]^-{\alpha_A} & & A \ar[rr]^-{\rho} \ar[d]^-{\alpha_A} & & H \otimes A \ar[d]_-{\alpha_H \otimes \alpha_A}\\
A \ar[rr]^-{\alpha_A} \ar@{<} `d[rrrr]  `[rrrr] _-{\rho_\alpha} [rrrr]
& & A \ar[rr]^-{\rho} & & H \otimes A.
}
\]
The upper-left square is commutative because $\alpha_A$ is a morphism of algebras.  The lower-right square is commutative by \eqref{eq:rhoalpha}.  The upper-right square is commutative because $\rho \colon A \to H \otimes A$ is a morphism of algebras \eqref{comodalgax}.
\end{proof}


We will use Theorem ~\ref{thm:comodalg} when $\rho \colon A \to H \otimes A$ is one of the three types of quantum symmetries on the quantum planes.  So we need suitable bialgebra morphisms on the quantum groups $\mq$ and $\mqns$ and algebra morphisms on the quantum planes.  We use the bialgebra morphisms $\alpha$ on $\mq$ and $\mqns$ defined as in \eqref{alphamq}.  Let us now describe the algebra morphisms on the quantum planes that are suitable for applying Theorem ~\ref{thm:comodalg}.

\begin{example}[\textbf{Hom-quantum planes}]
\label{ex:hqp}
Fix two scalars $\xi$ and $\lambda$ in $\bk$ with $\lambda$ invertible.
\begin{enumerate}
\item
There is an algebra morphism $\alpha \colon \qp \to \qp$ on the standard quantum plane \eqref{qp} determined by
\begin{equation}
\label{alphaqp}
\alpha\xymat = \alphaxy = \twistedxy.
\end{equation}
Indeed, $\alpha$ is a well-defined algebra morphism because it preserves the quantum commutation relation, i.e.,
\[
\begin{split}
\alpha(y)\alpha(x) &= (\lambda^{-1}\xi y)(\xi x)\\
&= \lambda^{-1}\xi^2 qxy\\
&= q\alpha(x)\alpha(y).
\end{split}
\]
The corresponding Hom-associative algebra (Example ~\ref{ex:homas})
\begin{equation}
\label{shqp}
\qpalpha = (\qp,\mu_\alpha = \alpha \circ \mu,\alpha)
\end{equation}
is called a \textbf{standard Hom-quantum plane}.  Since $\alpha$ depends on $\xi$ and $\lambda \not= 0$, we can think of $\qpalpha$ as a $2$-parameter, infinite family of Hom-associative algebra twistings of the standard quantum plane.  With respect to the $\bk$-linear basis $\{x^iy^j\}$ of $\qp$, the twisted multiplication $\mu_\alpha$ is given by
\begin{equation}
\label{hqpmu}
\mu_\alpha(x^iy^j, x^ky^l) = q^{jk}\lambda^{-(j+l)}\xi^{i+j+k+l} x^{i+k}y^{j+l}.
\end{equation}
Indeed, we have
\[
\begin{split}
\mu_\alpha(x^iy^j, x^ky^l)
&= \alpha((x^iy^j)(x^ky^l))\\
&= \alpha(q^{jk}x^{i+k}y^{j+l})\\
&= q^{jk}(\xi^{i+k}x^{i+k})(\lambda^{-(j+l)}\xi^{j+l}y^{j+l}).
\end{split}
\]
The definitions of $\alpha(x)$ and $\alpha(y)$ as in \eqref{alphaqp} are made so that the computation \eqref{rhoalphacom} below can go through.
\item
The formula \eqref{alphaqp} also defines an algebra morphism $\alpha \colon \qf \to \qf$ on the fermionic quantum plane \eqref{qf} because it preserves the three defining relations in $\qf$.  The corresponding Hom-associative algebra (Example ~\ref{ex:homas})
\begin{equation}
\label{fhqp}
\qfalpha = (\qf,\mu_\alpha = \alpha \circ \mu,\alpha)
\end{equation}
is called a \textbf{fermionic Hom-quantum plane}.  As above, we think of $\qfalpha$ as a $2$-parameter, infinite family of Hom-associative algebra twistings of the fermionic quantum plane.  With respect to the $\bk$-linear basis $\{1,x,y,xy\}$, the twisted operation $\mu_\alpha$ is given by the following multiplication table for $\mu_\alpha(a,b)$:
\begin{center}
\begin{tabular}{c|cccc}
$a \diagdown b$ & $1$ & $x$ & $y$ & $xy$ \\ \hline
$1$ & $1$ & $\xi x$ & $\lambda^{-1}\xi y$ & $\lambda^{-1}\xi^2 xy$\\
$x$ & $\xi x$ & $0$ & $\lambda^{-1}\xi^2xy$ & $0$ \\
$y$ & $\lambda^{-1}\xi y$ & $-q^{-1}\lambda^{-1}\xi^2xy$ & $0$ & $0$ \\
$xy$ & $\lambda^{-1}\xi^2xy$ & $0$ & $0$ & $0$
\end{tabular}
\end{center}
\item
Similarly, the formula \eqref{alphaqp} defines an algebra morphism $\alpha \colon \qm \to \qm$ on the mixed quantum plane \eqref{qm} because it preserves the two defining relations in $\qm$.  The corresponding Hom-associative algebra (Example ~\ref{ex:homas})
\begin{equation}
\label{mhqp}
\qmalpha = (\qm,\mu_\alpha = \alpha \circ \mu,\alpha)
\end{equation}
is called a \textbf{mixed Hom-quantum plane}.  As above, we think of $\qmalpha$ as a $2$-parameter, infinite family of Hom-associative algebra twistings of the mixed quantum plane.  The twisted multiplication $\mu_\alpha$ is given by the formula \eqref{hqpmu} (with $j,l \leq 1$) because $x$ and $y$ in $\qm$ also satisfy the quantum commutation relation, $yx = qxy$.
\qed
\end{enumerate}
\end{example}


We are now ready to describe Hom-quantum symmetries on the Hom-quantum planes.  Let us first consider the standard Hom-quantum planes \eqref{shqp}.

\begin{example}[\textbf{Hom-quantum symmetries on the standard Hom-quantum planes}]
\label{ex:shqsym}
Fix two scalars $\xi$ and $\lambda$ in $\bk$ with $\lambda$ invertible. Consider the standard Hom-quantum plane $\qpalpha$ \eqref{shqp} with $\alpha$ given as in \eqref{alphaqp}.  Recall the cobraided Hom-bialgebra $\mqalpha$ \eqref{mqalpha}, where $\alpha \colon \mq \to \mq$ is the bialgebra morphism defined in \eqref{alphamq}:
\[
\alpha\abcd = \abcdlambda.
\]
Also recall the $\mq$-comodule algebra structure $\rho \colon \qp \to \mq \otimes \qp$ on $\qp$ defined in \eqref{mqcoaction}:
\[
\rho\xymat = \abcd \otimes \xymat.
\]
We claim that Theorem ~\ref{thm:comodalg} applies in this situation, i.e., \eqref{eq:rhoalpha} holds.

Since both $\rho \circ \alpha$ and $(\alpha \otimes \alpha) \circ \rho$ are algebra morphisms, to check that they are equal it suffices to consider the algebra generators $x$ and $y$ in $\qp$.  We have
\begin{equation}
\label{rhoalphacom}
\begin{split}
\rho \circ \alpha\xymat
&= \begin{pmatrix} \xi\rho(x)\\ \lambda^{-1}\xi\rho(y)\end{pmatrix}\\
&= \begin{pmatrix} \xi(a \otimes x + b \otimes y)\\ \lambda^{-1}\xi(c \otimes x + d \otimes y)\end{pmatrix}\\
&= \abcdlambda \otimes \twistedxy\\
&= \alpha\abcd \otimes \alpha\xymat\\
&= (\alpha \otimes \alpha)\circ \rho\xymat.
\end{split}
\end{equation}
Therefore, by Theorem ~\ref{thm:comodalg}, the standard Hom-quantum plane $\qpalpha$ is a $\mqalpha$-comodule Hom-algebra with structure map $\rho_\alpha = \rho \circ \alpha \colon \qp \to \mq \otimes \qp$.  This is a $2$-parameter ($\xi$ and $\lambda \not= 0$) family of comodule Hom-algebra twistings of the usual $\mq$ quantum symmetry \eqref{mqcoaction} on the standard quantum plane $\qp$.

We now compute the twisted $\mqalpha$-coaction $\rho_\alpha = \rho \circ \alpha$ on $\qpalpha$ with respect to the $\bk$-linear basis $\{x^iy^j \colon i,j \geq 0\}$ of $\qp$.  By the computation \eqref{rhoalphacom}, we have
\begin{equation}
\label{rhoalphax}
\rho_\alpha\xymat = \abcdlambda \otimes \twistedxy.
\end{equation}
To compute $\rho_\alpha$ on the rest of $\qpalpha$, we use the $q$-symbols:
\begin{equation}
\label{qsymbols}
\begin{split}
(n)_{q^2} &= 1 + q^2 + q^4 + \cdots + q^{2(n-1)},\quad (m)_{q^2}! = (m)_{q^2}(m-1)_{q^2} \cdots (1)_{q^2},\quad (0)_{q^2}! = 1,\\
\binom{n}{r}_{q^2} &= \frac{(n)_{q^2}!}{(r)_{q^2}!(n-r)_{q^2}!}.
\end{split}
\end{equation}
It can be checked by direct computations \cite[IV.2.2 and IV.7.2]{kassel} that the quantum commutation relation \eqref{qcom} and the relations \eqref{six} in $\mq$ imply
\begin{equation}
\label{rhoxy}
\rho(x^iy^j) = \sum_{\substack{0\leq r \leq i\\0 \leq s \leq j}} q^{(i-r)s}\binom{i}{r}_{q^2} \binom{j}{s}_{q^2} a^rb^{i-r}c^sd^{j-s} \otimes x^{r+s}y^{i+j-r-s}.
\end{equation}
We have $\rho_\alpha = \rho \circ \alpha$, and
\begin{equation}
\label{alphaxiyj}
\alpha(x^iy^j) = (\xi x)^i(\lambda^{-1}\xi y)^j = \lambda^{-j}\xi^{i+j}x^iy^j.
\end{equation}
Combining \eqref{rhoxy} and \eqref{alphaxiyj}, we obtain
\begin{equation}
\label{rhoalphaxy}
\begin{split}
\rho_\alpha(x^iy^j)
&= \rho(\alpha(x^iy^j))\\
&= \lambda^{-j}\xi^{i+j}\rho(x^iy^j)\\
&= \sum_{\substack{0\leq r \leq i\\0 \leq s \leq j}} q^{(i-r)s}\lambda^{-j}\xi^{i+j} \binom{i}{r}_{q^2} \binom{j}{s}_{q^2} a^rb^{i-r}c^sd^{j-s} \otimes x^{r+s}y^{i+j-r-s}.
\end{split}
\end{equation}

Finally, as discussed in Example ~\ref{ex:mq}, the cobraiding form $R$ \eqref{mqR} on $\mq$ is fixed by $\alpha$ (i.e., $R = R \circ \alpha^{\otimes 2}$) regardless of what $\lambda \not= 0$ is.  Since we have $\rho \circ \alpha = (\alpha \otimes \alpha) \circ \rho$, we can apply Corollary ~\ref{cor2:comod} to obtain a solution $B_\alpha$ \eqref{Balphacomod} of the HYBE for $(\qp,\alpha)$.
\qed
\end{example}

\begin{example}[\textbf{Hom-quantum symmetries on the fermionic Hom-quantum planes}]
\label{ex:fhqsym}
Consider the $\mq$-comodule algebra structure $\rho \colon \qf \to \mq \otimes \qf$ \eqref{mqcoaction} on the fermionic quantum plane $\qf$ \eqref{qf} (Example ~\ref{ex:qsym}).

Fix two scalars $\xi$ and $\lambda$ in $\bk$ with $\lambda$ invertible. Recall from Example ~\ref{ex:hqp} that the formula \eqref{alphaqp} also defines an algebra morphism $\alpha \colon \qf \to \qf$.  With $\alpha \colon \mq \to \mq$ as in \eqref{alphamq}, the computation \eqref{rhoalphacom} implies that $\rho \circ \alpha = (\alpha \otimes \alpha)\circ \rho$.  Therefore, by Theorem ~\ref{thm:comodalg}, the fermionic Hom-quantum plane $\qfalpha$ \eqref{fhqp} is a $\mqalpha$-comodule Hom-algebra with structure map $\rho_\alpha = \rho \circ \alpha \colon \qf \to \mq \otimes \qf$.  This is a $2$-parameter ($\xi$ and $\lambda \not= 0$) family of comodule Hom-algebra twistings of the usual $\mq$ quantum symmetry on the fermionic quantum plane $\qf$.

On the algebra generators $x$ and $y$, the twisted $\mqalpha$-coaction $\rho_\alpha$ is given as in \eqref{rhoalphax}.  Since $\qf$ has a $\bk$-linear basis $\{1,x,y,xy\}$, in order to compute $\rho_\alpha$ completely, it remains to compute $\rho_\alpha(xy)$.  We claim that
\begin{equation}
\label{rhoalphaqf}
\rho_\alpha(xy) = \lambda^{-1}\xi^2 \detq \otimes xy,
\end{equation}
where $\detq = ad - q^{-1}bc$ is the quantum determinant \eqref{detq}.  Since $\alpha(xy) = \lambda^{-1}\xi^2xy$ (by \eqref{alphaxiyj} with $i = j = 1$), to prove \eqref{rhoalphaqf} it suffices to show
\[
\rho(xy) = \detq \otimes xy.
\]
We prove this in the following computation, using the relations $x^2 = y^2 = xy + qyx = 0$ in $\qf$:
\[
\begin{split}
\rho(xy) &= \rho(x)\rho(y) = (a \otimes x + b \otimes y)(c \otimes x + d \otimes y)\\
&= ad \otimes xy + bc \otimes yx\\
&= ad \otimes xy + bc \otimes (-q^{-1}xy)\\
&= (ad - q^{-1}bc) \otimes xy.
\end{split}
\]

Finally, we can apply Corollary ~\ref{cor2:comod} to obtain a solution $B_\alpha$ \eqref{Balphacomod} of the HYBE for $(\qf,\alpha)$.
\qed
\end{example}

\begin{example}[\textbf{Hom-quantum symmetries on the mixed Hom-quantum planes}]
\label{ex:mhqsym}
Consider the mixed quantum plane $\qm$ \eqref{qm} and the non-standard quantum group $\mqns$ (Example ~\ref{ex:mq'}).  Recall from Example ~\ref{ex:qsym} that the formula \eqref{mqcoaction} also defines a $\mqns$-comodule algebra structure $\rho \colon \qm \to \mqns \otimes \qm$ on $\qm$.

Fix two scalars $\xi$ and $\lambda$ in $\bk$ with $\lambda$ invertible.
With the bialgebra morphism $\alpha \colon \mqns \to \mqns$ \eqref{alphamq} and the algebra morphism $\alpha \colon \qm \to \qm$ \eqref{alphaqp}, Theorem ~\ref{thm:comodalg} applies by the computation \eqref{rhoalphacom}.  Thus, the mixed Hom-quantum plane $\qmalpha$ \eqref{mhqp} is a $\mqnsalpha$-comodule Hom-algebra with structure map $\rho_\alpha = \rho \circ \alpha$.  This is a $2$-parameter ($\xi$ and $\lambda \not= 0$) family of comodule Hom-algebra twistings of the usual $\mqns$ quantum symmetry on the mixed quantum plane $\qm$.  We can also apply Corollary ~\ref{cor2:comod} to obtain a solution $B_\alpha$ \eqref{Balphacomod} of the HYBE for $(\qm,\alpha)$.

Let us compute the twisted $\mqnsalpha$-coaction $\rho_\alpha = \rho \circ \alpha$ on $\qmalpha$ with respect to the $\bk$-linear basis $\{x^i,x^iy \colon i \geq 0\}$ of $\qm$.  We claim that
\begin{equation}
\label{rhoalphaxiy}
\begin{split}
\rho_\alpha(x^i) &= \xi^i\left\{a^i \otimes x^i + (i)_{q^2} a^{i-1}b \otimes x^{i-1}y\right\},\\
\rho_\alpha(x^iy) &= \lambda^{-1}\xi^{i+1}\left\{a^ic \otimes x^{i+1} +  (q(i)_{q^2} a^{i-1}bc + a^id) \otimes x^iy\right\},
\end{split}
\end{equation}
where $(i)_{q^2}$ is defined in \eqref{qsymbols}.  Since $\rho_\alpha = \rho \circ \alpha$, $\alpha(x^i) = \xi^ix^i$, and $\alpha(x^iy) = \lambda^{-1}\xi^{i+1}x^iy$ (by \eqref{alphaxiyj} with $j = 0,1$), to prove \eqref{rhoalphaxiy} it suffices to prove
\begin{equation}
\label{rhoxiy}
\begin{split}
\rho(x^i) &= a^i \otimes x^i + (i)_{q^2} a^{i-1}b \otimes x^{i-1}y,\\
\rho(x^iy) &= a^ic \otimes x^{i+1} +  (q(i)_{q^2} a^{i-1}bc + a^id) \otimes x^iy.
\end{split}
\end{equation}
To prove the first equality in \eqref{rhoxiy}, we compute as follows:
\[
\begin{split}
\rho(x^i) &= (a \otimes x + b \otimes y)^i\\
&= a^i \otimes x^i + (b \otimes y)(a^{i-1} \otimes x^{i-1}) + (a \otimes x)(b \otimes y)(a^{i-2} \otimes x^{i-2})\\
& \relphantom{} + \cdots + (a^{i-1} \otimes x^{i-1})(b \otimes y)\\
&= a^i \otimes x^i + (q^{2(i-1)} + q^{2(i-2)} + \cdots + q^2 + 1)a^{i-1}b \otimes x^{i-1}y\\
&= a^i \otimes x^i + (i)_{q^2} a^{i-1}b \otimes x^{i-1}y.
\end{split}
\]
In the second equality above, we used the relations $y^2 = 0$ and $yx = qxy$ in $\qm$.  In the third equality, we used the relations $ba = qab$ and $yx = qxy$.

Likewise, to prove the second equality in \eqref{rhoxiy}, we compute as follows:
\[
\begin{split}
\rho(x^iy)
&= \rho(x^i)\rho(y)\\
&= \left(a^i \otimes x^i + (i)_{q^2} a^{i-1}b \otimes x^{i-1}y\right)\left(c \otimes x + d \otimes y\right)\\
&= a^ic \otimes x^{i+1} + (i)_{q^2}a^{i-1}bc \otimes x^{i-1}yx + a^id \otimes x^iy.
\end{split}
\]
Now we obtain the second equality in \eqref{rhoxiy} by applying the relation $yx = qxy$ to the previous line.
\qed
\end{example}



\begin{thebibliography}{AA}
\bibitem{abe}
E. Abe, Hopf algebras, Cambridge Tracts in Math. 74, Cambridge Univ. Press, Cambridge, UK, 1977.

\bibitem{ama}
F. Ammar and A. Makhlouf, Hom-Lie superalgebras and Hom-Lie admissible superalgebras, arXiv:0906.1668v1.

\bibitem{artin2}
E. Artin, Theorie der Z\"{o}pfe, Abh. Math. Sem. Univ. Hamburg 4 (1925) 47-72.

\bibitem{artin}
E. Artin, Theory of braids, Ann. Math. 48 (1947) 101-126.

\bibitem{ams}
H. Ataguema, A. Makhlouf, and S. Silvestrov, Generalization of $n$-ary Nambu algebras and beyond, arXiv:0812.4058v1.


\bibitem{baxter}
R.J. Baxter, Partition function for the eight-vertex lattice model, Ann. Physics 70 (1972) 193-228.

\bibitem{baxter2}
R.J. Baxter, Exactly solved models in statistical mechanics, Academic Press, London, 1982.





\bibitem{cg}
S. Caenepeel and I. Goyvaerts, Hom-Hopf algebras, arXiv:0907.0187.

\bibitem{cp}
V. Chari and A.N. Pressley, A guide to quantum groups, Cambridge Univ. Press, Cambridge, 1994.

\bibitem{dmmz}
E.E. Demidov, Yu.I. Manin, E.E. Mukhin, and D.V. Zhdanovich, Non-standard quantum deformations of $GL(n)$ and constant solutions of the Yang-Baxter equations, Prog. Theo. Phys. Suppl. 102 (1990) 203-218.

\bibitem{dri83}
V.G. Drinfel'd, Hamiltonian structures on Lie groups, Lie bialgebras and the geometric meaning of the classical Yang-Baxter equations, Sov. Math. Dokl. 268 (1983) 285-287.

\bibitem{dri85}
V.G. Drinfel'd, Hopf algebras and the quantum Yang-Baxter equation, Soviet Math. Dokl. 32 (1985) 254-258.

\bibitem{dri87}
V.G. Drinfel'd, Quantum groups, in: Proc. ICM (Berkeley, 1986), p.798-820, AMS, Providence, RI, 1987.


\bibitem{es}
P. Etingof and O. Schiffmann, Lectures on quantum groups, 2nd ed., Int. Press of Boston, Cambridge, 2002.






\bibitem{fg}
Y. Fr\'{e}gier and A. Gohr, On Hom type algebras, arXiv:0903.3393v1.

\bibitem{fg2}
Y. Fr\'{e}gier and A. Gohr, On unitality conditions for hom-associative algebras, arXiv:0904.4874v1.

\bibitem{gohr}
A. Gohr, On hom-algebras with surjective twisting, arXiv:0906.3270v2.

\bibitem{hls}
J.T. Hartwig, D. Larsson, and S.D. Silvestrov, Deformations of Lie algebras using $\sigma$-derivations, J. Algebra 295 (2006) 314-361.

\bibitem{hay}
T. Hayashi, Quantum groups and quantum determinants, J. Algebra 152 (1992) 146-165.

\bibitem{hu}
N. Hu, $q$-Witt algebras, $q$-Lie algebras, $q$-holomorph structure and representations, Alg. Colloq. 6 (1999) 51-70.





\bibitem{jl}
Q. Jin and X. Li, Hom-Lie algebra structures on semi-simple Lie algebras, J. Algebra 319 (2008) 1398-1408.

\bibitem{gjw}
N. Jing, M.-L. Ge, and Y.-S. Wu, A new quantum group associated with ``non-standard" braid group representation, Lett. Math. Phys. 21 (1991) 193-203.


\bibitem{kassel}
C. Kassel, Quantum groups, Grad. Texts in Math. 155, Springer-Verlag, New York, 1995.


\bibitem{kr}
P.P. Kulish and N.Yu. Reshetikhin, Quantum linear problem for the sine-Gordon equation and higher representations, J. Soviet Math. 23 (1983) 2435-2441.

\bibitem{lt}
R.G. Larson and J. Towber, Two dual classes of bialgebras related to the concepts of ``quantum group" and ``quantum Lie algebra", Comm. Algebra 19 (1991) 3295-3345.

\bibitem{larsson}
D. Larsson, Global and arithmetic Hom-Lie algebras, Uppsala Univ. UUDM Report 2008:44, http://www.math.uu.se/research/pub/preprints.php.

\bibitem{ls}
D. Larsson and S.D. Silvestrov, Quasi-hom-Lie algebras, central extensions and $2$-cocycle-like identities, J. Algebra 288 (2005) 321-344.

\bibitem{ls2}
D. Larsson and S.D. Silvestrov, Quasi-Lie algebras, Contemp. Math. 391 (2005) 241-248.

\bibitem{ls3}
D. Larsson and S.D. Silvestrov, Quasi-deformations of $\mathfrak{sl}_2(\mathbb{F})$ using twisted derivations, Comm. Algebra 35 (2007) 4303-4318.


\bibitem{liu}
K. Liu, Characterizations of quantum Witt algebra, Lett. Math. Phy. 24 (1992) 257-265.

\bibitem{lm}
V.V. Lyubashenko and S. Majid, Fourier transform identities in quantum mechanics and the quantum line, Phys. Lett. B 284 (1992) 66-70.

\bibitem{majid91}
S. Majid, Representations, duals and quantum doubles of monoidal categories, Rend. Circ. Math. Palermo (2) Suppl. 26 (1991) 3061-3073.

\bibitem{majid92}
S. Majid, Anyonic quantum groups, in: Z. Oziewicz, et al. ed., Spinors, twistors, Clifford algebras and quantum deformations (Proc. 2nd Max Born Symp., Wroclaw, 1992), pp.327-336, Kluwer.

\bibitem{majid}
S. Majid, Foundations of quantum group theory, Cambridge U. Press, Cambridge, UK, 1995.

\bibitem{ms}
A. Makhlouf and S. Silvestrov, Hom-algebra structures, J. Gen. Lie Theory Appl. 2 (2008) 51-64.

\bibitem{ms2}
A. Makhlouf and S. Silvestrov, Hom-Lie admissible Hom-coalgebras and Hom-Hopf algebras, in: S. Silvestrov et. al. eds., Gen. Lie theory in Math., Phys. and Beyond, Ch. 17, pp. 189-206, Springer-Verlag, Berlin, 2009.
 
\bibitem{ms3}
A. Makhlouf and S. Silvestrov, Notes on formal deformations of Hom-associative and Hom-Lie algebras, to appear in Forum Math., arXiv:0712.3130v1.

\bibitem{ms4}
A. Makhlouf and S. Silvestrov, Hom-algebras and Hom-coalgebras, arXiv:0811.0400v2; Preprints in Math. Sci., Lund Univ., Center for Math. Sci., 2008.


\bibitem{mont}
S. Montgomery, Hopf algebras and their actions on rings, CBMS Regional Conf. Ser. in Math. 82, Amer. Math. Soc., Providence, R.I, 1993.


\bibitem{perk}
J.H.H. Perk and H. Au-Yang, Yang-Baxter equations, arXiv:math-ph/0606053v1.

\bibitem{rft}
N.Yu. Reshetikhin, L.A. Takhtadjian, and L.D. Faddeev, Quantization of Lie groups and Lie algebras, Leningrad Math. J. 1 (1990) 193-225.

\bibitem{rs}
L. Richard and S. Silvestrov, A note on quasi-Lie and Hom-Lie structures of $\sigma$-derivations of $\mathbb{C}[z_1^{\pm 1}, \ldots , z_n^{\pm 1}]$, in: S. Silvestrov et. al. eds., Gen. Lie theory in Math., Phys. and Beyond, Ch. 22, pp. 257-262, Springer-Verlag, Berlin, 2009.


\bibitem{sch}
P. Schauenburg, On coquasitriangular Hopf algebras and the quantum Yang-Baxter equation, Algebra Berichte 67, Verlag Reinhard Fischer, M\"{u}nchen, 1992.

\bibitem{ss}
G. Sigurdsson and S. Silvestrov, Lie color and Hom-Lie algebras of Witt type and their central extensions, in: S. Silvestrov et. al. eds., Gen. Lie theory in Math., Phys. and Beyond, Ch. 21, pp. 247-255, Springer-Verlag, Berlin, 2009.

\bibitem{skl1}
E.K. Sklyanin, On complete integrability of the Landau-Lifshitz equation, LOMI preprint E-3-1979, Leningrad, 1979.

\bibitem{skl2}
E.K. Sklyanin, The quantum version of the inverse scattering method, Zap. Nauchn. Sem. LOMI 95 (1980) 55-128.

\bibitem{skl3}
E.K. Sklyanin, On an algebra generated by quadratic relations, Uspekhi Mat. Nauk 40:2 (242) (1985) 214.

\bibitem{street}
R. Street, Quantum groups: a path to current algebra, Australian Math. Soc. Lecture Series 19, Cambridge Univ. Press, Cambridge, 2007.

\bibitem{sweedler}
M.E. Sweedler, Hopf algebras, Benjamin, New York, 1969.

\bibitem{yang}
C.N. Yang, Some exact results for the many-body problem in one dimension with replusive delta-function interaction, Phys. Rev. Lett. 19 (1967) 1312-1315.

\bibitem{yau}
D. Yau, Enveloping algebras of Hom-Lie algebras, J. Gen. Lie Theory Appl. 2 (2008) 95-108.

\bibitem{yau2}
D. Yau, Hom-algebras and homology, arXiv:0712.3515v2.

\bibitem{yau3}
D. Yau, Hom-bialgebras and comodule algebras, arXiv:0810.4866.

\bibitem{yau4}
D. Yau, Module Hom-algebras, arXiv:0812.4695v1.

\bibitem{yau5}
D. Yau, The Hom-Yang-Baxter equation, Hom-Lie algebras, and quasi-triangular bialgebras, J. Phys. A 42 (2009) 165202 (12pp).

\bibitem{yau6}
D. Yau, The Hom-Yang-Baxter equation and Hom-Lie algebras, arXiv:0905.1887.

\bibitem{yau7}
D. Yau, The classical Hom-Yang-Baxter equation and Hom-Lie bialgebras, arXiv:0905.1890.

\bibitem{yau8}
D. Yau, Hom-quantum groups I: quasi-triangular Hom-bialgebras, 	 arXiv:0906.4128v1.

\end{thebibliography}
\end{document}